\documentclass[oneside]{amsart}
\usepackage[utf8]{inputenc}
\usepackage{amsmath,amsthm,amsfonts,amssymb}
\usepackage[all]{xy}
\usepackage{color}
\usepackage{xparse}
\usepackage{aliascnt}
\usepackage[colorlinks,citecolor=blue,linkcolor=blue,urlcolor=blue,filecolor=blue]{hyperref}
\usepackage{bbm}
\usepackage{bbold}
\usepackage{enumerate}
\usepackage{tikz-cd}

\theoremstyle{definition}
\newtheorem{definition}{Definition}[section]
\newtheorem{remark}[definition]{Remark}
\newtheorem*{example}{Example}

\newtheorem{lemma}[definition]{Lemma}
\newtheorem{theorem}[definition]{Theorem}
\newtheorem{theoremx}{Theorem}

\newtheorem{prop}[definition]{Proposition}

\usepackage[mathscr]{euscript}

\author{Isaac Moselle}
\address{Centre for Mathematical Sciences\\ University of Cambridge\\ Wilberforce Rd\\Cambridge, CB3 0WA\\ UK}
\email{im472@cam.ac.uk}
\title{Homological stability for Iwahori-Hecke algebras of type $B_n$}


\begin{document}
\maketitle
\begin{abstract}
    We show that homological stability holds for the family of Iwahori-Hecke algebras of type $B_n$, where homology is identified with the relevant Tor group. This family of algebras is related to the Coxeter groups of type $B_n$, which are groups of signed permutations. The result builds on Hepworth's result for Iwahori-Hecke algebras of type $A_n$.
\end{abstract}
\section{Introduction}

\subsection{Homological stability}
A family of groups 
\[ G_0 \hookrightarrow{} G_1 \hookrightarrow{} G_2 \hookrightarrow{} \ldots\]
is said to exhibit \emph{homological stability} if for each $d$ the map
$\text{H}_d(G_{n-1}) \longrightarrow{} \text{H}_d(G_{n})$ is an isomorphism for $n$ sufficiently large. Prominent examples of sequences satisfying this include the symmetric groups $\mathfrak{S}_n$ (\cite{nakaoka1960decomposition}), automorphism groups of free groups Aut($F_n$) (\cite{hatcher1995homological}), and the linear groups $\text{GL}_n(R)$ over a ring (\cite{van1980homology}). A more complete survey of such results can be found in \cite{wahl2022homological}.

Recall that the homology of a group $G$ with coefficients in a commutative ring $R$ may be identified with $\text{Tor}_*^{RG}(\mathbb{1}, \mathbb{1})$ for $RG$ the group ring and $\mathbb{1}$ the trivial module ($R$ with trivial action). If $A$ is any $R$-algebra with a natural choice of trivial module $\mathbb{1}$, we may then define the \emph{algebra homology} of $A$ to be $\text{Tor}_*^{A}(\mathbb{1}, \mathbb{1})$.

\begin{definition}
A family of algebras $A_1 \hookrightarrow A_2 \hookrightarrow \ldots$ exhibits homological stability if for each $d$ the \emph{stabilisation map}
\[ \text{Tor}_d^{A_{n-1}}(\mathbb{1}, \mathbb{1}) \xrightarrow{} \text{Tor}_d^{A_{n}}(\mathbb{1}, \mathbb{1})\]
is an isomorphism for $n$ sufficiently large.
\end{definition}
\begin{remark}
For each $A_n$ we have picked a trivial module. We require also that the map $A_{i-1} \hookrightarrow A_i$ commutes with the action of each algebra on its trivial module. In practice, the choice of trivial module for the algebras we will consider is natural enough that they all evidently commute with the map. 
\end{remark}

In this paper we discuss this phenomenon in the context of \emph{Iwahori-Hecke algebras} of type $B_n$. Iwahori-Hecke algebras (to be defined in full later) are certain deformations of the group algebra of Coxeter groups, and are prevalent both in algebra and topology. They are used to define the \verb|HOMFLY-PT| polynomial in knot theory (\cite{jones1987hecke}), which is a generalisation of several knot polynomial invariants. They also arise naturally in the study of the representation theory of both $\text{GL}_n(\mathbb{F}_q)$ (\cite{mathas1999iwahori}) and Coxeter groups (\cite{kazhdan1979representations}). The representation theory of Iwahori-Hecke algebras is much studied (\cite{geck}), and has connections with quantum groups.

The family of Iwahori-Hecke algebras that we study are associated with the Coxeter group of type $B_n$ (also known as a \emph{hyperoctahedral} group), henceforth referred to as $\mathcal{B}_n$ (so that we may later distinguish between the group and the associated Coxeter diagram). The discussion of $\mathcal{B}_n$ as a Coxeter group appears in Section \ref{background-section}, but it will be more productive throughout to identify it with the group of \emph{signed permutations}.

\begin{definition}
Let $n\geq 1$. The group of signed permutations is the subgroup of $GL_n(\mathbb{Z})$ comprising of matrices with exactly one non-zero entry in each row and column, and with all non-zero entries $\in \{ \pm 1 \}$. Alternatively, it is the wreath product of $C_2$ with $\mathfrak{S}_n$.
\end{definition}

Elements of $\mathcal{B}_n$ may thus be thought of as permutations with an accompanying sign change in $(\mathbb{Z}/2\mathbb{Z})^n$, and the same notations and manipulations can be used as for $\mathfrak{S}_n$. This provides a convenient combinatorial viewpoint, some of which carries over to the associated Iwahori-Hecke algebras $\mathcal{HB}_n$ (defined in full later). Following the combinatorial viewpoint as a guide, we prove the algebras $\mathcal{HB}_n$ satisfy homological stability:

\begin{theoremx}\label{main-thm}
The map 
\[ \text{Tor}_d^{\mathcal{HB}_{n-1}}(\mathbb{1}, \mathbb{1}) \longrightarrow{}\text{Tor}_d^{\mathcal{HB}_{n}}(\mathbb{1}, \mathbb{1}) \]
is an isomorphism for $2d \leq n-1$.

\end{theoremx}

This is accomplished by constructing a chain complex $\mathcal{D}^\pm(n)$ which is the algebraic analogue of a combinatorial complex one could use to prove homological stability for the signed permutation groups. Most of this paper is devoted to the proof of the following theorem:

\begin{theoremx}\label{Dpm(n)-acyclic}
$H_d(\mathcal{D}^\pm (n)) = 0$ for $d \leq n-2$.
\end{theoremx}

As in the standard proof of homological stability for $\mathfrak{S}_n$ (\cite[3.1, pg 77]{maazen1979homology}), this highly acyclic complex is then used in a spectral sequence argument to obtain Theorem \ref{main-thm}.

\subsection{Recent related work}
This notion of homological stability for algebras was first introduced by Hepworth in \cite{Hepworth2020}. Here the result is obtained for the family $\mathcal{H}_n$, the simplest family of finite rank Iwahori-Hecke algebras---these are associated to the symmetric group $\mathfrak{S_n}$. 

In addition, a number of other results for algebra homology following the above definition have recently appeared in the literature. Boyd and Hepworth (\cite{boyd2020homology}) showed homological stability for the \emph{Temperley-Lieb} algebras $\text{TL}_n(a)$ and computed the stable homology to show that $H_d(\text{TL}_n(a)) = 0$ in a certain range and under certain conditions on the parameter $a \in R$ (which Randal-Williams later removed \cite{randalremark}). Sroka (\cite{sroka2022homology}) further investigated the homology of these algebras, using a generalisation of the Davis complex from the geometric study of Coxeter groups and showed that when $n$ is odd the homology of~$\text{TL}_n(a)$ vanishes in all degrees. Boyd, Hepworth and Patzt (\cite{boyd2021homology}) compared the homology of the \emph{Brauer} algebras $\text{Br}_n(\delta)$ to that of the symmetric group, and showed that they are isomorphic under some conditions. The algebras $\text{TL}_n(a)$ and $\text{Br}_n(\delta)$ are both examples of diagram algebras and are related to the group algebra of the braid group and the Iwahori-Hecke algebra $\mathcal{H}_n$.



$\mathcal{HB}_n$ by contrast is a quotient of the algebra of the Artin group associated to $\mathcal{B}_n$, which includes $R\mathfrak{S}_{n-1}$ as a subalgebra. The presence of the symmetric group in all these examples allows similar techniques to be used in each case---in particular, the face maps that appear in our complex $\mathcal{D}^\pm (n)$ are closely related to elements of the symmetric group.

\subsection{Outline}
In Section 2, we begin with an exposition of background information on Coxeter groups and Iwahori-Hecke algebras. In Section 3 we present a certain proof of homological stability for the signed permutation group $\mathcal{B}_n$ using the complex of \emph{signed injective words}, which motivates the proof of homological stability for $\mathcal{HB}_n$. Following this, we construct the analogous chain complex $\mathcal{D}^\pm(n)$ of $\mathcal{HB}_n$-modules. In Section 4, we study the double coset structure of the groups $\mathcal{B}_n$ and obtain some technical results regarding distinguished coset representatives. In Section 5, we define the filtration of $\mathcal{D}^\pm(n)$ in terms of these coset representatives. In Section 6, we identify the filtration quotients and show that they are highly connected, thus proving Theorem \ref{Dpm(n)-acyclic}. We then conclude in Section 7 via a routine spectral sequence argument (restated for algebra homology in \cite{Hepworth2020}).

\subsection{Acknowledgements}
I would like to thank my supervisor Rachael Boyd for suggesting the problem and guiding my work, along with providing patient support and constant motivation. This work was done as part of the CMS Summer Research Program at the University of Cambridge, and I am indebted to them for creating this opportunity. I would also like to thank the anonymous referee for their highly detailed and insightful comments.

\section{Background on Iwahori-Hecke algebras}\label{background-section}

This section summarises the essential theory of Coxeter groups and Iwahori-Hecke algebras, and sets out the conventions used in this paper. The results on Coxeter groups are standard (e.g.~ \cite{davis2012geometry}). The material on Iwahori-Hecke algebras follows the presentation in \cite{geck}, and is again standard. Much of the following is is similar to the background in \cite{Hepworth2020}.

\subsection{Coxeter groups}
A Coxeter group is a group $\mathcal{W}$ with presentation:
\[ \mathcal{W} = \Big\langle \ S \ \big | \ s^2=1, 
\ \langle s, t \rangle ^{m_{s,t}} = \langle t, s \rangle ^{m_{s,t}} \ \forall s\neq t \in S\Big\rangle\]
where $S$ is a finite set, $m_{s,t}=m_{t,s}\in \mathbb{N}_{\geq 2} \cup \{ \infty \}$, and $\langle s, t \rangle^m$ denotes the alternating product $stst\ldots$ of length $m$ starting with $s$ (no relation is applied for $m_{s,t}=\infty$). The relations $\langle s, t \rangle ^{m_{s,t}} = \langle t, s \rangle ^{m_{s,t}}$ are known as \textit{braid relations}, and the pair $ (\mathcal{W},S)$ the \textit{Coxeter system}.

Unpacking the braid relations, we see that $s, t \in S$ commute if $m_{s,t}=2$. In other cases:
\[ sts\ldots \quad = \quad tst\ldots \]
where both sides have length $m_{s,t}$. By the relation $s^2=t^2=1$, this is equivalent to:
\[ (st)^{m_{s,t}} = 1 \]

\begin{example}
For $S=\{s_1,s_2,\ldots, s_{n-1} \}$, set:
\begin{equation*}
    m_{s_i, s_j} = 
    \begin{cases}
      3 & \text{if} \ |i-j| = 1 \\
      2 & \text{if} \ |i-j| > 1
    \end{cases}
\end{equation*}
Then for $\mathcal{W}$ the associated Coxeter group, the relations are $s_i^2=1$, $s_is_{i+1}s_i = s_{i+1}s_is_{i+1}$ and $s_i$ commutes with $s_j$ for $|i-j| > 1$---this is the presentation for the symmetric group in $n$ letters $\mathfrak{S}_n$ generated by the $n-1$ transpositions of adjacent elements. This is also known as the Coxeter group of type $A_{n-1}$ (there are $n-1$ generators in $S$). We will refer to the generating set $S$ as $A_{n-1}$, and the group $\mathcal{W}$ as $\mathcal{A}_{n-1}$ - in general we use Roman letters to denote generating sets and calligraphic letters to denote the corresponding group.
\end{example}

In order to depict pictorially the relations in a Coxeter group, a Coxeter diagram is often used. This is a graph with vertex set $S$, and labelled edges assigned by the following rules:
\begin{itemize}
    \item If $m_{s,t} = 2$, so that $s, t$ commute in $\mathcal{W}$, no edge is drawn between them.
    \item If $m_{s,t} = 3$, an unlabelled edge is drawn.
    \item If $m_{s,t} > 3$, an edge is drawn with label $m_{s, t}$.
\end{itemize}

\begin{example}
Referring back to the previous example, the Coxeter group of type $A_{n-1}$, which is isomorphic to $\mathfrak{S}_n$, has Coxeter diagram:

\centerline{	\xymatrix@R=3mm {	
		& \textrm{} & & & \textrm{}\\	
		\hspace{10mm} &
		\begin{tikzpicture}[scale=0.15, baseline=0]
		\draw[fill= black] (5,0) circle (0.5);
		\draw[line width=1] (5,0) -- (15,0);
		\draw[fill= black] (15,0) circle (0.5);
		\draw[line width=1] (15,0) -- (25,0);
		\draw[fill= black] (25,0) circle (0.5);
		\draw (31,0) node {$\ldots \ldots$};
		\draw[fill= black] (37,0) circle (0.5);
		\draw[line width=1] (37,0) -- (47,0);
		\draw[fill= black] (47,0) circle (0.5);
		\draw (5,-1.5) node {$s_1$};
		\draw (15,-1.5) node {$s_2$};
		\draw (25,-1.5) node {$s_3$};
		\draw (37,-1.5) node {$s_{n-2}$};
		\draw (47,-1.5) node {$s_{n-1}$};
		\end{tikzpicture} \\
		}}
\vspace{10mm}
In particular, no edges are labelled, since the only values of $m_{s,t}$ that occur in the presentation of $\mathcal{A}_n$ are $2$ and $3$. By a slight abuse of notation, we will also refer to this diagram as $A_{n-1}$.

\end{example}

If the graph in the Coxeter diagram is not connected, then $\mathcal{W}$ is given by the direct product of the Coxeter groups associated to the connected components---in this case $\mathcal{W}$ is called reducible. Irreducible Coxeter groups, corresponding to connected Coxeter diagrams, may have finite or infinite order. A foundational result in the study of Coxeter groups is the classification of finite irreducible Coxeter groups.

\begin{theorem}[Classification of finite Coxeter groups, \cite{coxeter1935complete}]
Every finite irreducible Coxeter group is given by one of the following presentations:

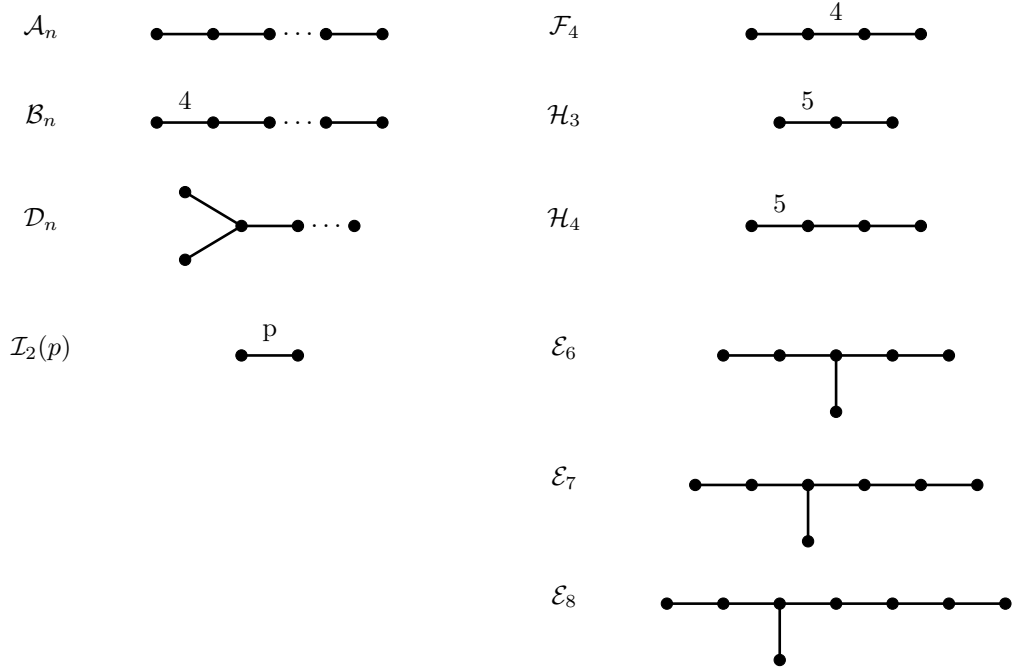
\resizebox{100mm}{!}{

\centerline{	\xymatrix@R=3mm {	
		& \textrm{} & & & \textrm{}\\	
		A_n & \begin{tikzpicture}[scale=0.15, baseline=0]
		\draw[fill= black] (5,0) circle (0.5);
		\draw[line width=1] (5,0) -- (15,0);
		\draw[fill= black] (10,0) circle (0.5);
		\draw[fill= black] (15,0) circle (0.5);		
		\draw (17.5,0) node {$\ldots$};
		\draw[fill= black] (20,0) circle (0.5);
		\draw[line width=1] (20,0) -- (25,0);				
		\draw[fill= black] (25,0) circle (0.5);
		\end{tikzpicture}
		& & F_4 & \begin{tikzpicture}[scale=0.15, baseline=0]
		\draw[fill= black] (5,0) circle (0.5);
		\draw[line width=1] (5,0) -- (20,0);
		\draw[fill= black] (10,0) circle (0.5);
		\draw[fill= black] (15,0) circle (0.5);
		\draw[fill= black] (20,0) circle (0.5);
		\draw (12.5,2) node {4};
		\end{tikzpicture}  \\
		B_n & \begin{tikzpicture}[scale=0.15, baseline=0]
		\draw[fill= black] (5,0) circle (0.5);
		\draw[line width=1] (5,0) -- (15,0);
		\draw[fill= black] (10,0) circle (0.5);
		\draw[fill= black] (15,0) circle (0.5);		
		\draw (17.5,0) node {$\ldots$};
		\draw[fill= black] (20,0) circle (0.5);
		\draw[line width=1] (20,0) -- (25,0);				
		\draw[fill= black] (25,0) circle (0.5);
		\draw (7.5,2) node {4};
		\end{tikzpicture}
		& & H_3 & \begin{tikzpicture}[scale=0.15, baseline=0]
		\draw[fill= black] (5,0) circle (0.5);
		\draw[line width=1] (5,0) -- (15,0);
		\draw[fill= black] (10,0) circle (0.5);
		\draw[fill= black] (15,0) circle (0.5);
		\draw (7.5,2) node {5};
		\end{tikzpicture}  \\
		D_n & \begin{tikzpicture}[scale=0.15, baseline=0]
		\draw[fill= black] (5,-3) circle (0.5);
		\draw[line width=1] (5,-3) -- (10,0);
		\draw[line width=1] (5,3) -- (10,0);
		\draw[line width=1] (15,0) -- (10,0);
		\draw[fill= black] (5,3) circle (0.5);
		\draw[fill= black] (10,0) circle (0.5);
		\draw[fill= black] (15,0) circle (0.5);
		\draw (17.5,0) node {$\ldots$};				
		\draw[fill= black] (20,0) circle (0.5);
		\end{tikzpicture}
		& & H_4 & \begin{tikzpicture}[scale=0.15, baseline=0]
		\draw[fill= black] (5,0) circle (0.5);
		\draw[line width=1] (5,0) -- (20,0);
		\draw[fill= black] (10,0) circle (0.5);
		\draw[fill= black] (15,0) circle (0.5);
		\draw[fill= black] (20,0) circle (0.5);
		\draw (7.5,2) node {5};
		\end{tikzpicture}  \\
		I_2(p) & \begin{tikzpicture}[scale=0.15, baseline=0]
		\draw[fill= black] (5,0) circle (0.5);
		\draw (7.5,2) node {p};
		\draw[line width=1] (5,0) -- (10,0);			
		\draw[fill= black] (10,0) circle (0.5);
		\end{tikzpicture}
		& &E_6 & \begin{tikzpicture}[scale=0.15, baseline=0]
		\draw[fill= black] (5,0) circle (0.5);
		\draw[line width=1] (5,0) -- (25,0);
		\draw[fill= black] (10,0) circle (0.5);
		\draw[fill= black] (15,0) circle (0.5);
		\draw[fill= black] (20,0) circle (0.5);		
		\draw[fill= black] (25,0) circle (0.5);			
		\draw[fill= black] (15,-5) circle (0.5);
		\draw[line width=1] (15,0) -- (15,-5);
		\end{tikzpicture}  \\
		& & & E_7 & \begin{tikzpicture}[scale=0.15, baseline=0]
		\draw[fill= black] (5,0) circle (0.5);
		\draw[line width=1] (5,0) -- (30,0);
		\draw[fill= black] (10,0) circle (0.5);
		\draw[fill= black] (15,0) circle (0.5);
		\draw[fill= black] (20,0) circle (0.5);		
		\draw[fill= black] (25,0) circle (0.5);		
		\draw[fill= black] (30,0) circle (0.5);		
		\draw[fill= black] (15,-5) circle (0.5);
		\draw[line width=1] (15,0) -- (15,-5);
		\end{tikzpicture}   \\
		& & & E_8 & \begin{tikzpicture}[scale=0.15, baseline=0]
		\draw[fill= black] (5,0) circle (0.5);
		\draw[line width=1] (5,0) -- (35,0);
		\draw[fill= black] (10,0) circle (0.5);
		\draw[fill= black] (15,0) circle (0.5);
		\draw[fill= black] (20,0) circle (0.5);		
		\draw[fill= black] (25,0) circle (0.5);		
		\draw[fill= black] (30,0) circle (0.5);		
		\draw[fill= black] (35,0) circle (0.5);		
		\draw[fill= black] (15,-5) circle (0.5);
		\draw[line width=1] (15,0) -- (15,-5);
		\end{tikzpicture} \\
}} }
\vspace{5mm}

\noindent where $A_n, B_n, D_n$ and $I_n$ are infinite families, and the remaining entries are the six exceptional groups.
\end{theorem}

In this paper, we are largely concerned with the groups $\mathcal{B}_n$ coming from the infinite family of diagrams $B_n$. We will label the first generator $u$, and the remaining generators $s_1,s_2,\ldots,s_{n-2}$. The diagram $A_{n-1}$ embeds in the tail of $B_n$, so that $s_1,\ldots,s_{n-2}$ generate a subgroup isomorphic to $\mathcal{A}_{n-1} \cong \mathfrak{S}_n$.

\centerline{	\xymatrix@R=3mm {	
		& \textrm{} & & & \textrm{}\\	
		\hspace{10mm} &
		\begin{tikzpicture}[scale=0.15, baseline=0]
		\draw[fill= black] (5,0) circle (0.5);
		\draw[line width=1] (5,0) -- (15,0);
		\draw[fill= black] (15,0) circle (0.5);
		\draw[line width=1] (15,0) -- (25,0);
		\draw[fill= black] (25,0) circle (0.5);
		\draw (31,0) node {$\ldots \ldots$};
		\draw[fill= black] (37,0) circle (0.5);
		\draw[line width=1] (37,0) -- (47,0);
		\draw[fill= black] (47,0) circle (0.5);
		\draw (5,-1.5) node {$u$};
		\draw (15,-1.5) node {$s_1$};
		\draw (25,-1.5) node {$s_2$};
		\draw (37,-1.5) node {$s_{n-3}$};
		\draw (47,-1.5) node {$s_{n-2}$};
        \draw (10,2) node {4};
		\end{tikzpicture} \\
		}}
\vspace{10mm}

We claimed in the introduction that $\mathcal{B}_n$ may be identified with the group of signed permutations on $n$ letters. Let us now describe an explicit map. The generator $u$ is sent to the matrix which sends $e_1$ to $-e_1$ and fixes $e_i$ for $i>1$, for $e_i$ a basis. The generators $s_i$ are sent to the standard permutation matrices that swap $e_i$ and $e_{i+1}$ and fix the other basis vectors. The relations among the generators in $\mathcal{B}_n$ are as follows:
\begin{itemize}
    \item $u^2=s_1^2=1$
    \item $us_1us_1=s_1us_1u$
    \item $us_i=s_iu$ for $i>1$
    \item $s_is_{i+1}s_i=s_{i+1}s_is_{i+1}$
    \item $s_is_j=s_js_i$ for $|i-j|>1$
\end{itemize}
It is easy to see that the matrices described obey these relations, and that the image of this map is the group of matrices with exactly one non-zero entry in each row and column and with all non-zero entries $\in \{ \pm 1 \}$. The map is in fact an isomorphism, as required, but this requires a slightly more detailed analysis. 

The usual cycle notation in $\mathfrak{S}_n$ may be extended to the group of signed permutations: for distinct $a_i$ and choice of signs $(-1)^{r_i}$, the cycle 

\[ \biggl((-1)^{r_1}a_1 \quad (-1)^{r_2}a_2 \quad \ldots \quad (-1)^{r_k}a_k \biggr) \]
represents the signed permutation which sends $+a_i$ (resp.~ $-a_i$) to $(-1)^{r_i + r_{i+1}}a_{i+1}$ (resp.~ $(-1)^{r_i + r_{i+1} + 1}a_{i+1}$). For example, the cycle $(2 \ -3)$ sends $2$ to $-3$, $-2$ to $3$, $-3$ to $2$, and $3$ to $-2$.

In this notation, the isomorphism with the group of signed permutations is given by associating $u$ with the cycle $(-1 \ 1)$, and $s_{i}$ with $(i \quad i+1)$. Also $\mathcal{B}_m$ embeds in $\mathcal{B}_n$ for $m\leq n$ by considering the signed permutations on $m$ letters inside the signed permutations on $n$ letters---in terms of generators, $\mathcal{B}_m$ is the subgroup generated by $u, s_1,\ldots,s_{m-1}$.

There are several important results in the theory of Coxeter groups that are used in this paper. Most significant are two theorems relating to the word problem in Coxeter groups. If $w \in \mathcal{W}$, then the length $l(w)$ is defined to be the smallest length of a word representing $w$ in the generators $S$. A word is reduced if it is a minimal representative, i.e.~ if $l(x_1x_2\ldots x_r) = r$.

\begin{theorem}[Matsumoto's theorem, \cite{matsumoto1964generateurs}]\label{matso}
For a Coxeter system $ (\mathcal{W},S)$, any two reduced words in $S$ for the same element $w$ are related by a series of transforms $\langle s, t \rangle ^{m_{s,t}} = \langle t, s \rangle ^{m_{s,t}}$.
\end{theorem}

From this, one can find a complete solution to the word problem.

\begin{theorem}[The word problem, \cite{tits1969probleme}]\label{word-problem}
For a Coxeter system $ (\mathcal{W},S)$, any two words represent the same element if and only if they are related by a series of transforms of the types
\begin{itemize}
    \item $s^2 \Rightarrow 1$
    \item $\langle s, t \rangle ^{m_{s,t}} \Rightarrow \langle t, s \rangle ^{m_{s,t}}$
\end{itemize}
These are known as \emph{M-moves}.
\end{theorem}

\subsection{Parabolic subgroups}
Any subset $J \subseteq S$ generates a subgroup $\mathcal{W}_J$ of $\mathcal{W}$. Such subgroups are known as \emph{parabolic}, and in the study of Coxeter groups it is found that $\mathcal{W}_J$ is a Coxeter group in a natural way: it forms the Coxeter system $(\mathcal{W}_J, J)$ with $m_{s,t}$ as in $(\mathcal{W},S)$. From the diagram perspective, every subset $J$ of the vertices generates a subgroup of $\mathcal{W}$ identified with the Coxeter group with diagram the the full diagram spanned by $J$.

Parabolic subgroups enjoy many special properties in the theory of Coxeter groups. In particular, there are distinguished coset representatives for each parabolic group---define:
\[ X_J = \{ w \in \mathcal{W} \, | \, l(sw) > l(w), \quad  \forall s \in J \} \]

This is the set of elements with no reduced expression beginning with a generator in $J$. These are known as $(J, \emptyset)$-reduced. By Matsumoto's Theorem (\ref{matso}), to check if a reduced word is $(J, \emptyset)$-reduced it suffices to show that no series of braid relations produces a word beginning with a letter in $J$. We may also consider the corresponding notion for words ending in an element of $X_J$: it is clear that the set of elements with no reduced word ending with a letter in $J$ is  $X_J^{-1}$.

The next theorem shows why $X_J$ are known as the \emph{distinguished right coset representatives}.

\begin{theorem}[{\cite[4.3.3]{davis2012geometry}}]
$X_J$ is a complete set of representatives for $\mathcal{W}_J\backslash \mathcal{W}$. Furthermore, for $x \in X_J$, $x$ is the shortest element in $\mathcal{W}_J x$.

Similarly, $X_J^{-1}$ is a complete set of representatives for $\mathcal{W} / \mathcal{W}_J$, and for $x \in X_J^{-1}$, $x$ is the shortest element in $x \mathcal{W}_J$.
\end{theorem}

\begin{example}
    As discussed, the permutation group $\mathfrak{S}_n \cong \mathcal{A}_{n-1}$ is a parabolic subgroup of $\mathcal{B}_n$ via the inclusion of the diagram $A_{n-1}$ into $B_n$. Then $X_{A_{n-1}}$ is the set of distinguished representatives for the right cosets $\mathcal{A}_{n-1}\backslash \mathcal{B}_n \cong \mathfrak{S}_n \backslash \mathcal{B}_n$.
\end{example}

The theory extends to double cosets. Take $X_{JK} = X_J \cap X_K^{-1}$, so $x\in X_{JK}$ are elements with no reduced words beginning with a letter in $J$, or ending with a letter in $K$. Analogous to the coset representatives, $X_{JK}$ is a complete set of representatives for $\mathcal{W}_J \backslash \mathcal{W} / \mathcal{W}_K$, and is made up of the shortest element in each double coset. The elements of $X_{JK}$ are called $(J, K)$-reduced, and are the \emph{distinguished double coset representatives} (\cite[2.1.7]{geck}).

It is often easier to describe the words that are $(J, K)$-reduced for some $K$ than those that are merely $(J, \emptyset)$-reduced. The \emph{Mackey decomposition} relates these:

\begin{prop}[Mackey decomposition, {\cite[2.1.9]{geck}}]\label{mackey}
For $J, K \subseteq S$:
\begin{equation*}
    X_J = \bigsqcup_{d\in X_{JK}} d \cdot X^K_{J^d \cap K}
\end{equation*}
where for $J' \subseteq K' \subseteq S$, $X^{K'}_{J'}$ denotes the distinguished coset representatives for $\mathcal{W}_{J'} \backslash \mathcal{W}_{K'}$, and $g^d = d^{-1}gd$ is conjugation with the positive power on the right.
\end{prop}

Inverting this produces the Mackey decomposition for right cosets:

\begin{equation*}
    X_J^{-1} = \bigsqcup_{d\in X_{KJ}} (X^K_{K \cap {}^d J })^{-1} \cdot d
\end{equation*}
where ${}^d g = dgd^{-1} $.

\subsection{Iwahori-Hecke algebras}
Given a Coxeter system $(\mathcal{W},S)$ and a commutative ring $R$ along with a choice of paramter $q \in R^\times$ then the associated \textit{Iwahori-Hecke algebra} is the algebra $\mathcal{HW}$ with generators $T_s$ for $s \in S$ and the following relations:
\begin{equation*}
    \begin{split}
        \langle T_s, T_t\rangle ^{m_{st}} &= \langle T_t, T_s\rangle ^{m_{st}} \\
        (T_s+1)(T_s - q) &= 0 
    \end{split} \qquad\qquad
    \begin{split}
        s \neq t &\in S \\
        s &\in S
    \end{split}
\end{equation*}
where $\langle T_s, T_t\rangle ^m = T_sT_tT_s\ldots$ is again the alternating product of length $m$ starting with $T_s$. The relations of the first type define the monoid algebra of the \emph{Artin monoid} associated to $(\mathcal{W},S)$, so $\mathcal{HW}$ is a quotient of the Artin monoid algebra. If $q=1$, the second relations become $T_s^2 = 1$, in which case $\mathcal{HW}$ is simply the group algebra of the Coxeter group $\mathcal{W}$.

Throughout we will fix $R, q$, and take $\mathcal{H}_n$ to be the Iwahori-Hecke algebra of type $A_{n-1} \cong \mathfrak{S}_n$, and $\mathcal{HB}_n$ to be the Iwahori-Hecke algebra of type $B_n$---the inclusion $A_{n-1}$ into $B_n$ via the tail of the Coxeter diagram induces an inclusion $\mathcal{H}_n \leq \mathcal{HB}_n$. This is the generalisation to Iwahori-Hecke algebras of the inclusion of $\mathfrak{S}_n$ into the group of signed permutations.

From the relation $(T_s+1)(T_s - q) = 0$ there are two natural rank one modules of $\mathcal{HW}$: $R$ where the generators act by multiplication by $q$, or multiplication by $-1$. The first of these becomes the trivial module when $q=1$, so will be denoted $\mathbb{1}$ and thought of as the trivial module for Iwahori-Hecke algebras. Then, as described earlier, the homology of the Iwahori-Hecke algebra is defined to be $\text{Tor}^{\mathcal{HW}}_*(\mathbb{1}, \mathbb{1})$ where the action of $T_s$ on $\mathbb{1} \cong R$ is multiplication by $q$.

\begin{remark}
The rank one module where all generators act by multiplication by $-1$ may be thought of as the analogue for the sign representation of $\mathfrak{S}_n$ in the Iwahori-Hecke setting. If the parameter $q$ is $-1$, it is equal to the trivial module $\mathbb{1}$.
\end{remark}
    
We proceed with some consequences of Matsumoto's Theorem (\ref{matso}).

\begin{prop}\label{matso-ih}
Let $w \in \mathcal{W}$ and $w=s_1s_2\ldots s_t = \Tilde{s}_1\Tilde{s}_2\ldots \Tilde{s}_r$ be two reduced expressions. Then $T_{s_1}T_{s_2}\ldots T_{s_t} = T_{\Tilde{s}_1}T_{\Tilde{s}_2}\ldots T_{\Tilde{s}_r}$ in $\mathcal{HW}$.
\end{prop}
\begin{proof}
By Matsumoto's Theorem, the two expressions are related in $\mathcal{W}$ by a series of substitutions $\langle s_i, s_j \rangle ^{m_{s_is_j}} = \langle s_j, s_i \rangle ^{m_{s_is_j}}$. But the relation $\langle T_{s_i}, T_{s_j} \rangle ^{m_{s_is_j}} = \langle T_{s_j}, T_{s_i} \rangle ^{m_{s_is_j}}$ holds in $\mathcal{HW}$ so the same substitutions show that $T_{s_1}T_{s_2}\ldots T_{s_t} = T_{\Tilde{s}_1}T_{\Tilde{s}_2}\ldots T_{\Tilde{s}_r}$.
\end{proof}

By the above result, it is valid to define $T_w = T_{s_1}T_{s_2}\ldots T_{s_t}$ for any $w\in \mathcal{W}$ where $s_1s_2\ldots s_t$ is a reduced expression. The \textit{basis theorem} asserts that these $T_w$ form a basis for $\mathcal{HW}$ over $R$.

\begin{theorem}[Basis theorem, {\cite[4.4.6]{geck}}]\label{basis-theorem}
The set $\{T_w : w\in \mathcal{W} \}$ is a basis for $\mathcal{HW}$ over $R$. In particular, $\mathcal{HW}$ is a free $R$-module.
\end{theorem}

As well as the $R$-module structure, $\mathcal{HW}$ is a left and right $\mathcal{HW}_J$-module for any subset $J \subseteq S$ of generators.

\begin{prop}\label{Hw-free-Hj}
For $(\mathcal{W},S)$ a Coxeter system and $J \subseteq S$, $\mathcal{HW}$ is a free left $\mathcal{HW}_J$-module with basis $\{ T_x : x \in X_J \}$. Hence $\mathbb{1} \otimes_{\mathcal{HW}_J} \mathcal{HW}$ is free with basis $\{1\otimes T_x : x \in X_J \}$.

Similarly, $\mathcal{HW}$ is a free right $\mathcal{HW}_J$-module with basis $\{ T_x : x \in (X_J)^{-1} \}$, and $\mathcal{HW} \otimes_{\mathcal{HW}_J} \mathbb{1}$ is free with basis $\{T_x \otimes 1 : x \in (X_J)^{-1} \}$.
\end{prop}

\subsection{Some notation} \label{notation}
We now define and discuss some notation for certain elements in Coxeter groups and Iwahori-Hecke algebras. The notation will also be introduced throughout the paper, so the reader may choose to ignore this section or return to it later.

The generators of the Coxeter group of type $B_n$ are throughout labelled as $u, s_1, s_2,\ldots,s_{n-1}$, where $u$ is the inital vertex in the Coxeter diagram. Similarly, the generators of the associated Iwahori-Hecke algebra $\mathcal{HB}_n$ are labelled $U, T_1, T_2,\ldots, T_{n-1}$, with $U=T_u$, and $T_i = T_{s_i}$.

\begin{itemize}
    \item For $n\geq a>b \geq 1$, take $s_{a,b}$ to be the element of $\mathfrak{S_n} \subseteq \mathcal{B}_n$ defined by:
    \[ s_{a,b} = s_{a-1}s_{a-2}\ldots s_b \]
    As a permutaion, this is
    \[ (a \quad a-1 \quad a-2  \quad \ldots \quad b+1 \quad b) \]
    the downwards cycle on $\{a, a-1, a-2, \ldots, b\}$.
    Observe that $s_{a,b}$ admits no M-moves, so by the solution the word problem it is reduced. We then define the following element of $\mathcal{HA}_{n-1} \subseteq \mathcal{HB}_n$:
    \[ T_{a,b} = T_{s_{a,b}} =  T_{a-1}T_{a-2}\ldots T_b \]
    where the last equality uses Proposition \ref{matso-ih}.
    
    \item For $m \leq n$, take $u_m$ to be the element of $\mathcal{B}_n$ defined by:
    \[ u_m = us_1s_2\ldots s_{m-1} \]
    In cycle notation, this is the product
    \[ (1 \quad -1) (1 \quad 2 \quad 3 \quad \ldots \quad m) \]
    so shifts $\{1, 2, 3, \ldots, m-1\}$ up by one, and sends $m$ to $-1$.
    Again, this admits no M-moves, so is reduced, and we define:
    \[ U_m = T_{u_m} = UT_1T_2\ldots T_{m-1} \]
    
    \item For $t\geq 0$, $m_1 > m_2 > \ldots > m_t$, define:
    \begin{equation*}
        v(m_1,m_2,\ldots,m_t) = v(\mathbf{m}) = u_{m_1}u_{m_2}\ldots u_{m_t} \\
    \end{equation*}
    
    \noindent From the discussion of $u_m$, $v(\mathbf{m})$ takes $m_1$ to $-1$, $m_2$ to $-2$ and so on, while sending the other letters to some positive letters. We shall see (Proposition \ref{vm-reduced}) that the $v(\mathbf{m})$ are reduced, so take:
    
    \[V(\mathbf{m}) = T_{v(\mathbf{m})} = U_{m_1}U_{m_2}\ldots U_{m_t} \]

\end{itemize}

\section{The complex of signed injective words}\label{complex-section}
In this section, we discuss some combinatorial complexes that relate to the group $\mathcal{B}_n$. The complex $\mathcal{C}^\pm(n)$ defined below is the combinatorial analogue to our algebraic complex $\mathcal{D}^\pm(n)$, and the proofs in this section motivate both our construction of $\mathcal{D}^\pm(n)$ and the approach to Theorem \ref{Dpm(n)-acyclic}. 

For $[n] = \{ 1,2,\ldots,n \}$, an injective word on $[n]$ is defined to be an ordered tuple $(a_0,\ldots,a_r)$ with $a_i \in [n]$ and $r < n$ such that no $a_i$ appears more than once.

\begin{definition}
Let $n\geq 0$. The \textit{complex of injective words} $\mathcal{C} (n)$ is the chain complex given in degree $r$ ($-1\leq r \leq n-1$) by the free $R$-module with basis the set of injective words $(a_0,\ldots, a_r)$ of length $r+1$ on $[n]$. $\mathcal{C}(n)_{-1}$ is a copy of $R$ generated by the empty word.

The differential $\partial^r: \mathcal{C}(n)_r \xrightarrow{} \mathcal{C} (n)_{r-1}$ is defined to be:
\[ \partial^r (a_0,\ldots, a_n) = \sum_{j=0}^r (-1)^j (a_0,\ldots, \widehat{ a_j}, \ldots, a_n)\]
where $\widehat{ a_j}$ means we omit the letter $a_j$ from the tuple.
\end{definition}

$\mathfrak{S}_n$ acts on the letters of each word in a natural way, so $\mathcal{C} (n)$ can be viewed as a chain complex of $R\mathfrak{S}_n$-modules. This complex is used in many proofs of homological stability for the permutation groups $\mathfrak{S}_n$, see for example \cite{maazen1979homology}, \cite{kerz2005complex}. In particular, the action of $\mathfrak{S}_n$ along with the following theorem is used in a Quillen spectral sequence argument:

\begin{theorem}\label{C(n)-acyclic}
\cite[2.1, pg 38]{maazen1979homology} $H_d(\mathcal{C} (n)) = 0$ for $d \leq n-2$.
\end{theorem}

For the group $\mathcal{B}_n$ of signed permutations, we introduce a new complex $\mathcal{C}^\pm (n)$ that will play an analogous role to  $\mathcal{C} (n)$. Define a signed injective word on $[n]$ to be an injective word on $[n]$ along with a choice of sign on each element---equivalently, an ordered tuple  $(\pm a_0,\ldots,\pm a_r)$ with $a_i \in [n]$ such that no $a_i$ appears more than once. Then the \textit{complex of signed injective words} $\mathcal{C}^\pm (n)$ is the complex given in degree $r$ ($-1 \leq r \leq n-1$) by the free $R$-module with basis the signed injective words of length $r+1$ on $[n]$, and the same differentials as $\mathcal{C} (n)$. $\mathcal{C}^\pm (n)$ carries an action of $\mathcal{B}_n$ (applying a signed permutation to each letter), so is a complex of $R\mathcal{B}_n$ modules.

We now prove the corresponding result to Theorem \ref{C(n)-acyclic}. 
The proof presented, while far from geodesic, will carry over to the setting of Iwahori-Hecke algebras and motivate the next sections of this paper, which deal with the proof of Theorem \ref{Dpm(n)-acyclic}.

\begin{theorem}\label{Cpm(n)-acyclic}
$H_d(\mathcal{C}^\pm (n)) = 0$ for $d \leq n-2$.
\end{theorem}

To accomplish this, we first define a filtration $F_p(n)$ of $\mathcal{C}^\pm(n)$ by the position of the negative elements in a given signed injective word.

\begin{definition}
For $0 \leq p \leq n$, let $F_p \subseteq {\mathcal{C}^\pm (n)}$ be the subcomplex spanned by the words in which all negatives appear in the first $p$ letters. Thus $F_0 = \mathcal{C}(n)$ is the complex of (unsigned) injective words, and 
\[\mathcal{C}(n) = F_0 \subseteq F_1 \subseteq \ldots \subseteq F_n = \mathcal{C}^\pm(n)\]
is a filtration of $\mathcal{C}^\pm(n)$. The $F_p$ are not $R\mathcal{B}_n$-submodules, but they are $R\mathfrak{S}_n$-submodules, since the action of $\mathfrak{S}_n$ does not change the signs of the letters.
\end{definition}

\begin{remark}
From the algebraic perspective that we will develop later, it is more natural to count positions from the right of the word, rather than the left. Unfortunately, this does not give a filtration in the Iwahori-Hecke setting, for technical reasons. In fact the main difficulty in this work was identifying the correct filtration: many of the more obvious filtrations of $\mathcal{C}^\pm (n)$ do not carry over to the Iwahori-Hecke setting.  
\end{remark}

Observe that by Theorem \ref{C(n)-acyclic}, we need only show that $F_p/F_{p-1}$ is highly acyclic. 

\begin{definition}
For $X$ a chain complex, take the suspension $\Sigma^k X$ to be the complex with $(\Sigma^k X)_r = X_{r-k}$, and $\partial^r_{\Sigma^k X} = \partial^{r-k}_X$.
\end{definition}

\begin{lemma}\label{C-quotient-lemma}
There is an isomorphism of chain complexes of $R\mathfrak{S}_n$ modules
\[ \Bigl(R\mathfrak{S}_n \otimes_{R\mathfrak{S}_{n-p}} \Sigma^p \mathcal{C}(n-p) \Bigr )^{\oplus 2^{p-1}} \longrightarrow F_p/F_{p-1}\]
\end{lemma}
\begin{proof}
Let $-1 \leq r \leq n - p - 1$. In degree $p+r$, $F_p$ is the free $R$-module with basis all words of the form 
\[ (\pm x_1, \pm x_2, \ldots, \pm x_p, y_0, y_1, \ldots, y_r)\]
(that is, all negative terms occuring in the first $p$ places). Thus $F_p/F_{p-1}$ is free with basis (in degree $p+r$) all words of the form
\[ (\pm x_1, \pm x_2, \ldots, - x_p, y_0, y_1, \ldots, y_r)\]
On such a word, the first $p$ face maps delete an $x_i$, meaning all negatives then lie in the first $p-1$ places (ie, the image lies in $F_{p-1}$). The differential in the quotient is thus:
\[ \partial^{p+r} (\pm x_1, \ldots, - x_p, y_0, \ldots, y_r) = (-1)^p \sum_{j=0}^r (-1)^j(\pm x_1, \ldots, - x_p, y_0, \ldots, \widehat{y_j}, \ldots ,y_r)\]

In particular, the signs of $x_1,\ldots,x_{p-1}$ are not changed by any face map. Thus for each $2^{p-1}$ choices of signs ($\mathbf{v} \in \{ \pm 1 \}^{p-1}$), there is a $R\mathfrak{S}_n$-subcomplex $M_{\mathbf{v}}$ of $F_p/F_{p-1}$ spanned in degree $p+r$ by all words of the form: 
\[ (\mathbf{v}_1 \cdot x_1, \ldots, \mathbf{v}_{p-1} \cdot x_{p-1}, -x_p, y_0,\ldots, y_r) \]
and there is a decomposition:
\[ F_p/F_{p-1} = \bigoplus_{\mathbf{v} \in \{ \pm 1 \}^{p-1}} M_\mathbf{v} \]
Consider now the map $\phi:  R\mathfrak{S}_n \otimes_{R\mathfrak{S}_{n-p}} \Sigma^p \mathcal{C}(n-p) \longrightarrow M_{\mathbf{v}}$ given in degree $p+r$ by
\[ \sigma \otimes (y_0,\ldots,y_r) \longmapsto \sigma (\mathbf{v}_1 \cdot (n-p+1),\ldots, \mathbf{v}_{p-1} \cdot (n-1), -n,y_0,\ldots,y_r) \]
This is a chain map by the above discussion of the differential in $F_p/F_{p-1}$ (we omit the leading factor $(-1)^p$ in $\partial^{p+r}$ above, as this does not affect homology). For each injective word $x = (x_1, \ldots, x_p)$ on $[n]$, pick $\sigma_x$ with $x = \sigma_x(n-p+1,\ldots,n)$. The $\sigma_x$ are a set of coset representatives for $\mathfrak{S}_n / \mathfrak{S}_{n-p}$, so every $\sigma \otimes (y_0,\ldots,y_r)$ may be written as $\sigma_x \otimes (y_0',\ldots,y_r')$ for unique $x$, $y_i'$. Then $\phi$ has inverse given by
\[ (\mathbf{v}_1 \cdot x_1, \ldots, \mathbf{v}_{p-1} \cdot x_{p-1}, -x_p, y_0,\ldots, y_r) \longmapsto \sigma_x \otimes \sigma_x^{-1} ( y_0,\ldots, y_r) \]

\noindent so $M_\mathbf{v} \cong R\mathfrak{S}_n \otimes_{R\mathfrak{S}_{n-p}} \Sigma^p \mathcal{C}(n-p)$ for all $\mathbf{v}$.
\end{proof}

The theorem now follows immediately.

\begin{proof}[Proof of Theorem \ref{Cpm(n)-acyclic}] $R\mathfrak{S}_n$ is free as a right $R\mathfrak{S}_{n-p}$ module, so the above lemma implies that:
\[ H_d(F_p/F_{p-1}) = R\mathfrak{S}_n \otimes_{R\mathfrak{S}_{n-p}} \Bigl (H_{d-p}(\mathcal{C}(n-p))\Bigr )^{\oplus 2^{p-1}} \]
which is $0$ for $d \leq n - 2$ by Theorem \ref{C(n)-acyclic}.\end{proof}

$\mathcal{C}^\pm (n)$ can now be used to produce a proof of homological stability for $\mathcal{B}_n$---this also follows from the result of Hepworth for homological stability for families of Coxeter groups \cite{hepworth2016homological}, or Hatcher and Wahl's result on homological stability for wreath products (Proposition 1.6 of \cite{hatcher2010stabilization}).

\subsection{The algebraic complex of injective words}
In order to translate the results of the previous discussion to the Iwahori-Hecke algebra $\mathcal{HB}_n$, we first provide an algebraic description of the complex $\mathcal{C}^\pm (n)$. A detailed discussion of this construction for the case of the complex of injective words $\mathcal{C} (n)$ can be found in Section 4 of \cite{Hepworth2020}.

\begin{definition}
Let the complex $\widetilde{\mathcal{C}}^\pm (n)$ be given in degree $r$ ($-1 \leq r \leq n-1$) by
\[\widetilde{\mathcal{C}}^\pm (n)_r = R\mathcal{B}_n\otimes_{R\mathcal{B}_{n-r-1}} \mathbb{1}\]
with differential $\partial^r: R\mathcal{B}_n\otimes_{R\mathcal{B}_{n-r-1}} \mathbb{1} \longrightarrow R\mathcal{B}_n\otimes_{R\mathcal{B}_{n-r}} \mathbb{1} $
\[\partial^r_j(\sigma \otimes 1) = \sigma s_{n-r+j, n-r} \otimes 1\]
\[\partial^r = \sum_{j=0}^r (-1)^j \partial^r_j\]
where $s_{a,b} = s_{a-1}s_{a-2}\ldots s_b$ for $s_i$ the generators of $\mathfrak{S}_n \subseteq \mathcal{B}_n$ as a Coxeter group.
\end{definition}

A signed injective word of length $r+1$ on $[n]$ may be thought of as a record of the image of $n-r,n-r+1,\ldots,n$ under a signed permutation $\sigma$. But the image of $n-r,n-r+1,\ldots,n$ under $\sigma$ is determined uniquely by the coset of $\sigma \mathcal{B}_{n-r-1}$, whence we obtain a bijection from the set of signed injective words to the cosets $\mathcal{B}_n/\mathcal{B}_{n-r-1}$. By standard properties of the group algebra, $R\mathcal{B}_n\otimes_{R\mathcal{B}_{n-r-1}} \mathbb{1}$ is the free $R$-module with basis $\sigma_x \otimes 1$ for $\{\sigma_x\}$ a set of coset representatives of $\mathcal{B}_n/\mathcal{B}_{n-r-1}$, so there is a natural identification $R\mathcal{B}_n/\mathcal{B}_{n-r-1} \cong R\mathcal{B}_n\otimes_{R\mathcal{B}_{n-r-1}} \mathbb{1}$.

To understand the face maps, we may observe that the projection $\mathcal{B}_n/\mathcal{B}_{n-r-1} \xrightarrow{ } \mathcal{B}_n/\mathcal{B}_{n-r}$ is given by forgetting the image of $n-r$, which under the above bijection corresponds to discarding the leftmost letters of an injective word. If we multiply on the right by the permutation $s_{n-r+j,n-r} = (n-r+j \quad n-r+j-1 \quad \ldots \ n-r)$ before the projection map, then the image of $n-r+j$ is instead forgotten, so the $j^{th}$ letter of the injective word is discarded, and all preceeding letters shifted up one space. This precisely the effect of deleting the $j^{th}$ element.

The next proposition simply formalises this discussion.

\begin{prop}\label{alg-complex}
The chain complex $\widetilde{\mathcal{C}}^\pm (n)$ is isomorphic to $\mathcal{C}^\pm (n)$.
\end{prop}
\begin{proof}
In degree $r$, let $\phi: \widetilde{\mathcal{C}}^\pm (n) \xrightarrow{} \mathcal{C}^\pm (n)$ be the map:
\[ \sigma \otimes 1 \longmapsto \sigma(n-r, n-r+1,\ldots, n)\]

$\widetilde{\mathcal{C}}^\pm (n)_r$ is free with basis $\sigma_x \otimes 1$ for $\{\sigma_x\}$ coset representatives of $\mathcal{B}_n/\mathcal{B}_{n-r-1}$, each of which is uniquely determined by the image of $(n-r,\ldots,n)$. $\mathcal{C}^\pm (n)$ has basis the set of signed injective words of length $r+1$ on $[n]$, thus $\phi_r$ is a bijection on basis elements, so is an isomorphism.

To see that $\phi$ is a chain map, we calculate:
\begin{align*}
    \partial^r_j (\phi_r (\sigma \otimes 1)) &= \partial^r_j(\sigma(n-r, n-r+1,\ldots, n))\\
     &= \partial^r_j(\sigma (n-r), \sigma (n-r+1), \ldots, \sigma(n)) \\
     &= (\sigma (n-r), \ldots, \widehat{ \sigma (n-r+j)}, \ldots, \sigma(n)) \\
     &= (\sigma(s_{n-r+j,n-r} (n-r+1)), \ldots, \sigma((s_{n-r+j,n-r}n)))\\
     &= \phi_{r-1}(\partial^r_j (\sigma \otimes 1))
\end{align*}
so $\phi$ commutes with the face maps, and hence the differentials.
\end{proof}

\subsection{The complex $\mathcal{D}^\pm (n)$}
In order to apply the Quillen spectral sequence argument to obtain a proof of Theorem \ref{main-thm}, we wish to construct a complex over the Iwahori-Hecke algebra $\mathcal{HB}_n$ that plays the same role as $\mathcal{C}^\pm (n)$ does for $R\mathcal{B}_n$. The algebraic description of $\mathcal{C}^\pm(n)$ indicates precisely what this should be.

\begin{definition}
Let the complex $\mathcal{D}^\pm(n)$ be given in degree ($-1 \leq r \leq n-1$) by
\[ \mathcal{D}^\pm(n)_r = \mathcal{HB}_n \otimes_{\mathcal{HB}_{n-r-1}} \mathbb{1} \]
with differential $\partial^r: R\mathcal{B}_n\otimes_{R\mathcal{B}_{n-r-1}} \mathbb{1} \longrightarrow R\mathcal{B}_n\otimes_{R\mathcal{B}_{n-r}} \mathbb{1} $
\[\partial^r_j(x \otimes 1) = x T_{n-r+j, n-r} \otimes 1\]
\[\partial^r = \sum_{j=0}^r (-1)^j q^{-j} \partial^r_j\]
where $T_{a,b} = T_{a-1}T_{a-2}\ldots T_b$ for $T_i$ the generators of $\mathcal{H}_n \subseteq \mathcal{HB}_n$.
\end{definition}
\begin{remark}
The powers of $q$ appearing in $\partial^r$ are necessary for this to be a chain complex; recall that the generators of $\mathcal{HB}_n$ act on $\mathbb{1}$ via multiplication by $q$, which introduces powers of $q$ to the computations.
\end{remark}

 Proposition \ref{alg-complex} shows that this is the same as the complex of signed injective words for $q=1$. $\mathcal{D}^\pm (n)$ is very similar to other constructions occuring in the literature around homological stability for diagram algebras (cf.~the induced modules $\text{TL}_n \otimes_{\text{TL}_m} \mathbb{1}$ in \cite{boyd2020homology} and $\text{Br}_n \otimes_{\text{Br}_m} \mathbb{1}$ in \cite{boyd2021homology}). In particular, the complex $\mathcal{D}(n)$ of \cite{Hepworth2020} (whence our naming and notation) is identical to $\mathcal{D}^\pm (n)$ with $\mathcal{H}_n$ in place of $\mathcal{HB}_n$. The fact that the differentials are well-defined and that $\partial^{r-1} \circ \partial^r = 0$ are Lemma 6.4 and 6.6 in \cite{Hepworth2020} (which demonstrates the necessity of the powers of $q$ in $\partial^r$), we do not repeat them here. 

To transport the proof of Theorem \ref{Cpm(n)-acyclic} to the Iwahori-Hecke setting, two things must be accomplished. First, an algebraic description of the filtration $F_p$ must be found---this is completed in Sections \ref{coset-reps} and \ref{Fp-section}. Secondly, the quotients $F_p/F_{p-1}$ must be identified, and a result corresponding to Lemma \ref{C-quotient-lemma} proved---this is Section \ref{Fp-quotients-section}. Some choices made in the proof of Lemma \ref{C-quotient-lemma} become less natural when looked at from the algebraic perspective, so the results in Section \ref{Fp-quotients-section} have a slightly different form. The overarching argument, however, is morally the same, and in particular the decomposition of $F_p/F_{p-1}$ into subcomplexes relating to $\mathcal{C}(n-p)$ (resp. $\mathcal{D}(n-p)$ of \cite{Hepworth2020}) still holds.

\section{Coset representatives for $\mathfrak{S}_n\backslash \mathcal{B}_n$}\label{coset-reps}
Observe that if $\mathcal{C}^\pm (n)_r$ is identified with $R\mathcal{B}_n\otimes_{R\mathcal{B}_{n-r-1}} \mathbb{1} \cong R[\mathcal{B}_n/\mathcal{B}_{n-r-1}] $, then the position of negatives is determined by the double coset $\mathfrak{S}_n \backslash \mathcal{B}_n / \mathcal{B}_{n-r-1}$. Indeed, two words have negatives appearing in the same places if and only if there is an (unsigned) permutation relating one to the other. Thus in order to determine the algebraic analogue of our filtration in the previous section, it will be instructive to examine these double cosets, and in particular the cosets $\mathfrak{S}_n \backslash \mathcal{B}_n$.

Recall that $X_{A_{n-1}}^{B_n}$ is the set of distinguished right coset representatives for $\mathcal{A}_{n-1} \cong \mathfrak{S}_n$ in $\mathcal{B}_n$. By the Mackey decomposition (Proposition \ref{mackey}) applied to $A_{n-1}, B_{n-1} \subseteq B_n$:

\begin{equation}\label{mackey-sn}
     X_{A_{n-1}}^{B_n} = \bigsqcup_{d \in X_{A_{n-1} B_{n-1}}} d\cdot X^{B_{n-1}}_{A_{n-1}^d \cap B_{n-1}  } 
\end{equation}

Hence we first find $X_{A_{n-1} B_{n-1}}$ (the distinguished double coset representatives $d$ for $\mathfrak{S}_n \backslash \mathcal{B}_n / \mathcal{B}_{n-1}$), then calculate $A_{n-1}^d \cap B_{n-1}$ (where the intersection refers only to that of generating sets) to reconstruct $X_{A_{n-1}}^{B_n}$.

Recall from Section \ref{notation} that for $m\leq n$, $u_m = us_1s_2\ldots s_{m-1} \in \mathcal{B}_n$ (with $u_m = u$ when $m=1$).

\begin{lemma}\label{sn-bn-b(n-1)}
The distinguished double coset representatives for $\mathfrak{S}_n \backslash \mathcal{B}_n / \mathcal{B}_{n-1}$ are
\[ X_{A_{n-1} B_{n-1}} = \{ 1, u_n \} \]
\end{lemma}
\begin{proof}
We first find the distinguished representatives for $\mathcal{B}_n / \mathcal{B}_{n-1}$. Consider the elements:
\begin{align*}
    s_ts_{t+1}\ldots s_{n-1} \qquad &t=1,\ldots,n\\
    s_{t}s_{t-1}\ldots s_1 u s_1 s_2 \ldots s_{n-1} \qquad &t=0,\ldots,n-1
\end{align*}
All elements in these two families admit no M-moves, so all are reduced and furthermore have unique reduced form (Theorem \ref{word-problem}). None end in an element in $\mathcal{B}_{n-1}$, so all are $(\emptyset, B_{n-1})$-reduced, hence they are distinguished coset representatives for $\mathcal{B}_n/\mathcal{B}_{n-1}$. $|\mathcal{B}_n/\mathcal{B}_{n-1}| = 2n$, so these are all the distinguished coset representatives (they are distinct by considering the length).

The distinguished double coset representatives for $\mathfrak{S}_n \backslash \mathcal{B}_n / \mathcal{B}_{n-1}$ are precisely the representatives for $\mathcal{B}_n/\mathcal{B}_{n-1}$ that are $(A_{n-1}, \emptyset)$-reduced, ie.~ do not begin with a generator of $\mathfrak{S}_n$. This is only $1$ and $us_1s_2\ldots s_{n-1} = u_n$. 

\end{proof}

Next we show that for $d=u_n$, the term $X^{B_{n-1}}_{A_{n-1}^d \cap \mathcal{B}_{n-1}}$ in Equation (\ref{mackey-sn}) is equal to $X^{A_{n-1}}_{A_{n-2}}$.

\begin{lemma}
Let $s_1,s_2,\ldots s_{n-1}$ be the generators in $A_{n-1}$, $u, s_1,\ldots , s_{n-2}$ the generators in $B_{n-1}$, and $u_n$ as above. In $\mathcal{B}_n$:
\[ \{s_1,s_2,\ldots, s_{n-1} \} ^{u_n} \cap \{ u, s_1,\ldots , s_{n-2} \} = \{ s_1, \ldots, s_{n-2}  \}\]
where $(-)^g$ denotes conjugation by $g$ with the positive power of $g$ appearing on the right.
\end{lemma}
\begin{proof}
Consider first $s_i$ with $i > 1$. $s_i$ commutes with $u$ and $s_j$ for $j<i-1$, and $u^2 = s_j^2 = 1$, so:
\begin{align*}
    s_i^{u_n} &= s_{n-1}\ldots s_2s_1u \cdot s_i \cdot us_1s_2\ldots s_{n-1} \\
    &= s_{n-1}\ldots s_is_{i-1} \cdot s_i \cdot s_{i-1}s_i\ldots s_{n-1} \\
    \intertext{By the braid relation:}
    &= s_{n-1}\ldots s_{i+1}s_i^2 \cdot s_{i-1} \cdot s_i^2s_{i+1}\ldots s_{n-1} \\
    &= s_{n-1}\ldots s_{i+1} \cdot s_{i-1} \cdot s_{i+1}\ldots s_{n-1} \\
    \intertext{$s_{i-1}$ commutes with all remaining terms:}
    &= s_{i-1}\cdot s_{n-1}\ldots s_{i+1}^2 \ldots s_{n-1} \\
    &= \ldots = s_{i-1}\cdot s_{n-1}^2 = s_{i-1}
\end{align*}
Thus $\{s_2,\ldots s_{n-1} \} ^{u_n} \cap \{ u, s_1,\ldots , s_{n-2} \} = \{ s_1, \ldots, s_{n-2} \}$. It remains only to show that $s_1^{u_n} \neq u$. But the image of $1$ under $s_1^{u_n} = s_{n-1}\ldots s_2s_1u \cdot s_1 \cdot us_1s_2\ldots s_{n-1}$ is $-n$, so it cannot be $u$.
\end{proof}

Now Equation (\ref{mackey-sn}) provides an inductive description of $X^{B_n}_{A_{n-1}}$.

\begin{prop}\label{vm-reduced}
Given $t \geq 0$, $n \geq m_1 >  m_2 >  \ldots >  m_t \geq 1$, define:
\begin{align*}
    v(\mathbf{m}) &= v(m_1,m_2,\ldots,m_t) = u_{m_1}u_{m_2}\ldots u_{m_t} \\
    &= us_1s_2\ldots s_{m_1 - 1}\cdot us_1s_2\ldots s_{m_2 - 1} \cdot \ldots \cdot us_1s_2\ldots s_{m_t - 1}
\end{align*}
so that $t=0$ forces the empty product $v(\mathbf{m})=1$. Then the elements $v(\mathbf{m})$ are the distinguished coset representatives for $\mathfrak{S}_n \backslash \mathcal{B}_n$:
\[ X^{B_n}_{A_{n-1}} = \{v(\mathbf{m}) :  t \geq 0, n \geq m_1 >  m_2 >  \ldots >  m_t \geq 1\} \]
\end{prop}
\begin{proof}
From Lemma \ref{sn-bn-b(n-1)}, $X_{A_{n-1} B_{n-1}} = \{ 1, u_n \}$. For $d = 1, u_n$, we have seen that $X^{B_{n-1}}_{A_{n-1}^d \cap B_{n-1}} = X^{B_{n-1}}_{A_{n-2}}$. Thus by Equation (\ref{mackey-sn}) (the Mackey decomposition):
\[ X_{A_{n-1}}^{B_n} = X^{B_{n-1}}_{A_{n-2}} \ \bigcup \ u_n \cdot X^{B_{n-1}}_{A_{n-2}} \]
By induction, we may assume the result for $n-1$ (in the base case of $n=1$, we have $\mathcal{B}_n= \{1,u\}$ and $\mathcal{A}_n = \{1\}$, so there is a single coset representative $u$, which agrees with the statement). Then:
\begin{align*}
    X_{A_{n-1}}^{B_n} &= \{v(\mathbf{m}) :  n-1 \geq m_1 >  \ldots >  m_t \geq 1\} \ \cup \ u_n \{v(\mathbf{m}) : n-1 \geq m_1 >  \ldots >  m_t \geq 1\} \\
    &= \{v(\mathbf{m}) :  n \geq m_1 >  \ldots >  m_t \geq 1\} \qedhere
\end{align*} 
\end{proof}
\begin{remark}
In this formulation, the distinguished coset representatives have been parameterised by decreasing sequences $m_1 >  m_2 >  \ldots >  m_t$ in $[n]$. There are $2^n$ such sequences, as expected: $|\mathfrak{S_n} \backslash \mathcal{B}_n| = 2^n n! / n! = 2^n$.
\end{remark}

The $v(\mathbf{m})$ defined above play a central role in the remainder of this proof. It is not hard to observe that the signed injective word of length $n$ corresponding to $v(\mathbf{m})$ has negative letters at precisely the places $m_i$ when counted from the left, so the algebraic analogue of $F_p$ will be based on the value of $m_1$, the leftmost negative.

We first complete our description of the distinguished representatives for $\mathfrak{S}_n \backslash \mathcal{B}_n / \mathcal{B}_{n-r-1}$, as discussed earlier.

\begin{prop}\label{reps-sn-bn-bnr-1}
The distinguished representatives $X_{A_{n-1}B_{n-r-1}}$ for $\mathfrak{S}_n \backslash \mathcal{B}_n / \mathcal{B}_{n-r-1}$ are given by 
$v(\mathbf{m}) = v(m_1,\ldots  ,m_t)$ for $n \geq m_1 >  \ldots  >  m_t \geq n-r$. Equivalently, $v(\mathbf{m} + n - r - 1)$ for $r+1 \geq m_1 >  m_2 >  \ldots m_t \geq 1$.
\end{prop}
\begin{proof}
These distinguished representatives for $\mathfrak{S}_n \backslash \mathcal{B}_n / \mathcal{B}_{n-r-1}$ are all $(A_{n-1}, \emptyset)$-reduced, so they form a subset of $X^{B_n}_{A_{n-1}}$. By Proposition \ref{vm-reduced}, the elements of $X_{A_{n-1}B_{n-r-1}}$ are then of the form $v(\mathbf{m})$ for certain $\mathbf{m}$. If $m_t < n-r$, then $v(\mathbf{m})$ ends with an element lying in $\mathcal{B}_{n-r-1}$, so is not $(A_{n-1}, B_{n-r-1})$-reduced. Hence $X_{A_{n-1}B_{n-r-1}}$ is contained in the set of $v(\mathbf{m})$ with $m_t \geq n-r$.

To see that it is in fact all such $v(\mathbf{m})$, note that $|\mathfrak{S}_n \backslash \mathcal{B}_n / \mathcal{B}_{n-r-1}|$ is equal to the number of orbits of the action of $\mathfrak{S}_n$ on the signed injective words of length $r+1$, by the earlier discussion equating signed injective words with cosets of $\mathcal{B}_{n-r-1}$. There are $2^{r+1}$ orbits (one for each pattern of signs). The number of sequences $\mathbf{m}$ with  $n \geq m_1 >  \ldots  >  m_t \geq n-r$ is $2^{r+1}$, so each $v(\mathbf{m})$ with $m_t\geq n-r$ must be a representative of a double coset (in particular it is $(A_{n-1}, B_{n-r-1})$-reduced).
\end{proof}

Each representative of $\mathfrak{S}_n \backslash \mathcal{B}_n / \mathcal{B}_{n-r-1}$ corresponds to a certain pattern of signs for injective words of length $r+1$. This proposition says that the natural choice of representative for words of length $r+1$ with negatives at the places $m_i$ from the left is $v(\mathbf{m} + n - r - 1)$.

Recall again that the algebraic analogue of signed injective words of length $r+1$ is given by $R\mathcal{B}_n\otimes_{R\mathcal{B}_{n-r-1}} \mathbb{1} \cong R[\mathcal{B}_n/\mathcal{B}_{n-r-1}]$. The filtration of $\mathcal{C}^\pm(n)_r$ was given by considering the orbits under the action of $\mathfrak{S}_n$; it is thus natural to try to understand how $R[\mathcal{B}_n/\mathcal{B}_{n-r-1}]$ splits under the action of $R\mathfrak{S}_n$. Again, the Mackey decomposition (Proposition \ref{mackey}) is used.

\begin{lemma}\label{right-cosets-bnr1}
The right cosets of $\mathcal{B}_{n-r-1}$ in $\mathcal{B}_n$ have distinguished representatives given by
\[(X_{B_{n-r-1}})^{-1} = \bigsqcup_{n \geq m_1 >  m_2 > \ldots >  m_t \geq n-r} (X^{A_{n-1}}_{\{s_{1+t},\ldots ,s_{n-r-2+t}\}})^{-1}v(m_1,\ldots  ,m_t)\]
\end{lemma}
\begin{proof}
By Proposition \ref{reps-sn-bn-bnr-1}, $X_{A_{n-1}B_{n-r-1}}$ is $\{v(\mathbf{m}) : n \geq m_1 >  m_2 >  \ldots  >  m_t \geq n-r\}$. The Mackey decomposition states that:
\[(X_{B_{n-r-1}})^{-1} = \bigsqcup_{d \in X_{A_{n-1}B_{n-r-1}}} (X^{A_{n-1}}_{A_{n-1} \cap {}^d B_{n-r-1} })^{-1}d\]
It only remains to calculate $(X^{A_{n-1}}_{A_{n-1} \cap {}^d B_{n-r-1} })^{-1}$. Given $s_k \in B_{n-r-1}$ ($k < n-r-1$), and $u_{m_t}$ with $m_t \geq n-r$, we can compute
\[ u_{m_t} s_k = (u s_1\ldots s_{m_t - 1} )s_k = s_{k+1}(u s_1\ldots s_{m_t - 1} ) = s_{k+1} u_{m_t}\]
and thus $u_{m_t}s_ku_{m_t}^{-1} = s_{k+1}$. Continuing in this fashion, after $l$ steps we have $(u_{m_{t-l+1}}\ldots  u_m)s_k(u_{m_{t-l+1}}\ldots  u_m)^{-1} = s_{k+l}$ (and $m_{t-l} \geq n-r+l  > k+l+1$, so the computation in each step is the same). Thus for $d=v(m_1,m_2,\ldots,m_t)$, $^ds_k = s_{k+t}$. As a signed permutation, $^du$ sends $t+1$ to $-(t+1)$, so is not $s_j$ for any $j$. Hence we have that $X^{A_{n-1}}_{A_{n-1} \cap {}^d B_{n-r-1} } = X^{A_{n-1}}_{\{s_{1+t},\ldots ,s_{n-r-2+t}\}}$, as required.
\end{proof}

As discussed earlier, the element $v(\mathbf{m})$ corresponds to a word with negative terms at the places $m_i$ from the left (and is furthermore the most natural choice of representative for this sign pattern from the Coxeter group point of view). Given a signed injective word of length $r$ with sign pattern $v(\mathbf{m})$, we may then pick some $\sigma \in \mathfrak{S}_n$ such that the coset of $\sigma v(\mathbf{m})$ in $\mathcal{B}_n/\mathcal{B}_{n-r-1}$ corresponds to this signed injective word. The above lemma makes precise to what extent this choice of $\sigma$ is unique, and provides natural choices for $\sigma$ in each case (a distinguished coset representative in $(X^{A_{n-1}}_{\{s_{1+t},\ldots ,s_{n-r-2+t}\}})^{-1}$).

Finally, we wish to extend this description to the complex $\mathcal{D}^\pm(n)$. The previous lemma, along with the basis theorem (Theorem \ref{basis-theorem}) for Iwahori-Hecke algebras, provides a convenient basis from which we will define our filtration.

\begin{prop}\label{Dn-decomp}
$(\mathcal{D}^\pm(n))_r = \mathcal{HB}_n \otimes_{\mathcal{HB}_{n-r-1}}\mathbb{1}$ is a free $R$ module with basis
\[\bigsqcup_{t\geq 0} \quad \bigsqcup_{n \geq m_1 > m_2 >\ldots > m_t \geq n-r}\{T_xV(\mathbf{m}) \otimes 1 : x \in (X^{A_{n-1}}_{\{s_{1+t},\ldots ,s_{n-r-2+t}\}})^{-1}\}\]
where 
\[V(\mathbf{m}) = T_{v(\mathbf{m})} = UT_1\ldots T_{m_1-1}\ldots  UT_1\ldots T_{m_t-1}\]
\end{prop}
\begin{proof}
The $v(\mathbf{m})$ are reduced by Proposition \ref{vm-reduced}, so equality in the second statement holds by Proposition \ref{matso-ih}. For the first statement, by Proposition \ref{Hw-free-Hj}, $(\mathcal{D}^\pm(n))_r$ is free with basis $\{T_x \otimes 1: x \in (X_{B_{n-r-1}})^{-1}\}$. By Lemma \ref{right-cosets-bnr1}, this is precisely the set of $T_{x\cdot v(\mathbf{m})}\otimes 1 =T_x V(\mathbf{m})\otimes 1$ (again by Proposition \ref{matso-ih}), for the $x \in (X^{A_{n-1}}_{\{s_{1+t},\ldots ,s_{n-r-2+t}\}})^{-1}$.
\end{proof}

\section{The algebraic filtration}\label{Fp-section}
In this section, we construct the filtration of $\mathcal{D}^\pm(n)$. This filtration will correspond to the filtration of $\mathcal{C}^\pm(n)$ by position of the negative letters. We have seen the most natural representatives for signed injective words with a certain pattern of signs are the elements $v(\mathbf{m})$, and that the corresponding elements of $\mathcal{HB}_n$ are $V(\mathbf{m})$. We thus define a filtration of $\mathcal{D}^\pm(n)$ in terms of the $\mathcal{H}_n$ span of certain $V(\mathbf{m}) \otimes 1$(recall that $\mathcal{H}_n \subseteq \mathcal{HB}_n$ is the Iwahori-Hecke algebra of type $A_{n-1} \cong \mathfrak{S}_n$).

Earlier we saw that the natural representative (for the $\mathfrak{S}_n$ action) of a word of length $r+1$ with negatives at places $m_1,m_2,\ldots,m_t$ (counted from the left) is $v(\mathbf{m} + n - r - 1)$; a word thus has all negatives in the first $p$ places if and only if it is in the $\mathfrak{S}_n$ orbit of $v(\mathbf{m} + n - r - 1)$ for some $\mathbf{m}$ with $m_1 \leq p$.

\begin{definition}
Let $1\leq p \leq n$. Define $F_p$ to be the subcomplex given in degree $r$ by the $\mathcal{H}_n$ span of the elements $V(\mathbf{m} + n  - r - 1 ) \otimes 1$ where $m_1 \leq \min \{r+1, p \}$:
\[ (F_p)_r = \mathcal{H}_n \cdot \{ V(\mathbf{m} + n  - r - 1 ) \otimes 1 : \min \{r+1, p \} \geq m_1 > m_2 > \ldots > m_t \geq 1 \} \]
\end{definition}

\noindent so that $F_0$ is generated over $\mathcal{H}_n$ by $1\otimes1$. $(F_p)_r$ is evidently a $\mathcal{H}_n$-submodule, and it will soon be shown that $F_p$ is a $\mathcal{H}_n$-subcomplex.

Before we show that $F_p$ is indeed a filtration, we identify the first level $F_0$---this showcases further analogy with the filtration of the complex of signed injective words. In the filtration of $\mathcal{C}^\pm(n)$ the first level of the filtration was isomorphic to a complex (the complex of injective words, $\mathcal{C}(n)$) known to be highly acyclic. The analogue of this for Iwahori-Hecke algebras is the complex $\mathcal{D}(n)$ introduced in \cite{Hepworth2020} (Definition 6.3). There it is shown that:

\begin{theorem}[\cite{Hepworth2020}]
$H_d(\mathcal{D}(n)) = 0$ for $d \leq n-2$.
\end{theorem}

Accordingly, we now identify $F_0$ with $\mathcal{D}(n)$ and thus deduce that it is highly acyclic.

\begin{prop}
$F_0$ is a subcomplex isomorphic as a chain complex of $\mathcal{H}_n$ modules to the complex $\mathcal{D}(n)$. In particular, $H_d(F_0) = 0$ for $d \leq n-2$.
\end{prop}
\begin{proof}
From the definition of $F_p$, $F_0$ is given in degree $r$ by the  $\mathcal{H}_n$ span of $1\otimes1$ in $\mathcal{HB}_n \otimes_{\mathcal{HB}_{n-r-1}}\mathbb{1}$. For $x \in \mathcal{H}_n$, $\partial^r_j(x \otimes 1) = xT_{n-r+j, n-r} \otimes 1 = xT_{n-r+j, n-r} (1 \otimes 1) \in F_0$, so $F_0$ is indeed a $\mathcal{H}_n$-subcomplex.

In degree $r$, $\mathcal{D}(n)$ is $\mathcal{H}_n \otimes_{\mathcal{H}_{n-r-1}} \mathbb{1}$, so there is an obvious map (of chain complexes of $\mathcal{H}_n$ modules) $\mathcal{D}(n) \xrightarrow{} \mathcal{D}^\pm(n)$. By Proposition \ref{Hw-free-Hj}, $\mathcal{H}_n \otimes_{\mathcal{H}_{n-r-1}}\mathbb{1}$ has a $R$-basis $\{T_x \otimes 1 : x \in (X^{A_{n-1}}_{A_{n-r-2}})^{-1}\}$. By Proposition \ref{Dn-decomp}, this is mapped to an $R$-basis of $(F_0)_r$, so this map is an isomorphism on to $F_0$.
\end{proof}

We now show that $F_p$ is a filtration. It will be convenient to first calculate the product of $U_m = UT_1T_2\ldots  T_{m-1}$ with certain $T_{a,b}$ (relating to the face maps in $\mathcal{D}^\pm(n)$).

\begin{lemma}\label{v-facemap-calc}
For $m \geq b$, $a \geq b, m\neq a$:
\begin{equation*}
    U_m\cdot T_{a,b} = 
    \begin{cases}
      T_{a+1,b+1}U_{m} & m>a>b \\ 
      T_{a,b+1} U_{m+1} & a>m\geq b\\
      U_m & a=b
    \end{cases}
\end{equation*}
\end{lemma}
\begin{proof}
If $a=b$, $T_{a,b} =1$ which gives the final case. We assume $a>b$. First suppose $m>a$:
\[U_mT_{a,b} = UT_1\ldots  T_{m-1}\cdot T_{a-1}\ldots  T_b\]
\[ = UT_1\ldots  T_{a-1}T_aT_{a-1}T_{a+1}\ldots  T_{m-1} \cdot T_{a-2}\ldots  T_{b}\]
\[ = UT_1\ldots  T_aT_{a-1}T_aT_{a+1}\ldots  T_{m-1} \cdot T_{a-2}\ldots  T_{b}\]
\[ = T_a\cdot UT_1\ldots  T_{a-1}T_aT_{a+1}\ldots  T_{m-1} \cdot T_{a-2}\ldots  T_{b}\]
\[ = T_a U_m T_{a-1, b}\]
so inductively $U_mT_{a,b} = T_aT_{a-1}\ldots  T_{b+1}U_m = T_{a+1, b+1} U_m$.

Now if $m<a$, we may write $T_{a,b} = T_{a,m+1}T_mT_{m,b}$. All terms in $T_{a,m+1}$ commute with $U_m$, since they are $T_k$, $k \geq m+1$, and $U_m$ involves $U, T_k$ for $k \leq m-1$. $U_mT_m = U_{m+1}$ by definition, so:
\[U_m T_{a,b} = U_mT_{a,m+1}T_mT_{m,b} = T_{a,m+1}U_mT_mT_{m,b}\]
\[ = T_{a,m+1}U_{m+1}T_{m,b} = T_{a,m+1}T_{m+1, b+1} U_{m+1}\]
\[ = T_{a,b+1}U_{m+1}\]
where the second last equality comes from the first case.
\end{proof}

Upon digesting the above lemma, the reader may notice that we have skipped the only case of the calculation ($m=a$) in which the quadratic relations of $\mathcal{HB}_n$ are used. Further, upon performing this case of the calculation, the reader will observe that the answer involves a sum of distinct terms with various coefficients. Not only do these terms lead to hopeless complications downstream, they are in fact the chief difficulty in trying to construct a filtration of $\mathcal{D}^\pm (n)$. It is commendable that our filtration $F_p$ allows us to skip this case.

\begin{prop}\label{fp-is-filtration}
$F_p$ is a subcomplex of $\mathcal{D}^\pm (n)$. Furthermore, $\partial_j^r  (F_p)_r \subseteq (F_{p-1})_{r-1}$ when $j < p$.
\end{prop}
\begin{proof}
The face maps are $\mathcal{H}_n$-linear, so it suffices to evaluate them on the generators $\{V(\mathbf{m} + n - 1 -r) \otimes 1 : \min\{r+1, p\} \geq m_1 >\ldots >m_t\geq 1\}$ of $(F_p)_r$. Suppose first that $j \leq m_1 - 1 < p$. Then $n-r+j \leq m_1 + n - 1 - r$ so $V(\mathbf{m} + n - 1 -r)T_{n-r+j, n-r}$ is contained in the subalgebra $\mathcal{HB}_{m_1 + n - 1 -r}$. Proposition \ref{Dn-decomp} says that (in particular) $\mathcal{HB}_{m_1 + n - 1 -r}$ is spanned over $\mathcal{H}_{m_1 + n - 1 -r}$ by $\{V(\mathbf{m}') : m_1' \leq m_1 + n - 1 - r\}$, so $V(\mathbf{m} + n - 1 -r)T_{n-r+j, n-r}$ is in this span. If $V(\mathbf{m}')$ ends in a factor $U_{m_t'}$ with $m_t' \leq n-r$, this factor may be moved over the tensor product in $\mathcal{HB}_n \otimes_{\mathcal{HB}_{n-r}} \mathbb{1}$, which results only in a number of multiplications by $q$. Hence we see that
\[\partial_j^r V(\mathbf{m} + n - 1 -r) \otimes 1 = V(\mathbf{m} + n - 1 -r)T_{n-r+j, n-r} \otimes 1\]
lies in the $\mathcal{H}_n$ span of $\{V(\mathbf{m}') \otimes 1 : m_1 + n - 1 - r \geq m_1' > \ldots   > m_t' \geq n-r+1\}$, which is equal to the set $\{V(\mathbf{m}' + n - 1 - (r-1)) \otimes 1 : m_1 - 1 \geq m_1' > \ldots   > m_t' \geq 1\}$. By definition of $F_p$, $m_1 \leq \min\{r+1, p\}$, so $m_1' \leq \min\{r, p-1\}$ and hence $\partial_j^r V(\mathbf{m} + n - 1 -r) \in (F_{p-1})_{r-1}$.

Now suppose $j \geq m_1$, so that $n-r+j > m_i + n -r -1$ for all $i$. Then we show by induction on $t$ that:
\[ V(\mathbf{m} + n -1 - r)T_{n-r+j, n-r} = T_{n-r+j, n-r+t}V(\mathbf{m} + n - 1 - (r-1)) \]
When $t=0$, $V(\mathbf{m} + n -1 - r)=V(\mathbf{m} + n - 1 - (r-1))=1$ and there is nothing to show. For $t>0$, by the second case of Lemma \ref{v-facemap-calc} (which we may apply since $n-r+j > m_t + n - r - 1 \geq n-r$):
\[ U_{m_t+n-1-r}T_{n-r+j,n-r} = T_{n-r+j, n-r+1}U_{m_t+n-r}\]
so then
\[V(\mathbf{m} + n -1 - r)T_{n-r+j, n-r} = V(m_1+n-1-r,\ldots  ,m_{t-1} + n - 1 -r)U_{m_t+n-1-r}T_{n-r+j,n-r}\]
\[ = V(m_1+n-1-r,\ldots  ,m_{t-1} + n - 1 -r)T_{n-r+j, n-r+1}U_{m_t+n-r} \]
We now apply the inductive hypothesis, after a suitable re-indexing: write $\mathbf{m}'=(m_1-1,\ldots,m_{t-1}-1)$, so that the last expression can be written as
\[ V(\mathbf{m}'+n-1-(r-1))T_{n-(r-1)+(j-1),n-(r-1)}U_{m_t+n-r} \]
Now $j-1\leq m'_1=m_1-1$, so by induction this is:
\[ = T_{n-r+j, n-r+t}V(m_1+n - r,\ldots  ,m_{t-1} + n -r)U_{m_t+n-r}\]
which is $T_{n-r+j, n-r+t}V(\mathbf{m} + n - 1 - (r-1))$.

Thus we have shown:
\[\partial_j^r V(\mathbf{m} + n - 1 -r) \otimes 1 = T_{n-r+j, n-r+t}V(\mathbf{m} + n - 1 - (r-1))\otimes 1\]
$r \geq j \geq m_1$, so $m_1 \leq \min\{r, p\}$, and hence this element lies in $(F_p)_{r-1}$, completing the proof that $F_p$ is a filtration. Also if $p > j \geq m_1$ then it is in $(F_{p-1})_{r-1}$, and consequently the latter statement holds.
\end{proof}

\begin{lemma}
$F_n$ is the entire complex $\mathcal{D}^\pm (n)$.
\end{lemma}
\begin{proof}
By definition, $\mathcal{D}^\pm (n)$ is $0$ outside of degrees $-1 \leq r \leq n-1$, so we may restrict attention to degrees $r<n$. In these degrees, $\min\{r+1, n\} = r+1$. Then in degree $r$, $F_n$ is the $\mathcal{H}_n \otimes 1$ span of $V(m_1 + n - 1 -r,\ldots ,m_t + n-1-r)$ for $r+1 \geq m_1 >\ldots  >m_t\geq 1$. If $\widetilde{\mathbf{m}} = \mathbf{m} + n - (r+1)$, then this is $V(\widetilde{m_1},\ldots  ,\widetilde{m_t}) \otimes 1$ for $n \geq \widetilde{m_1} > \ldots   > \widetilde{m_t} \geq n-r$, which spans $(\mathcal{D}^\pm (n))_r$ over $\mathcal{H}_n$ by Proposition \ref{Dn-decomp}.
\end{proof}

\section{Identifying the filtration quotients}\label{Fp-quotients-section}
In this section, we identify the quotients $F_p/F_{p-1}$ and show that they are highly acyclic. This proceeds in three steps: first a description of basis elements and the chain maps in $F_p/F_{p-1}$ is obtained, and it is shown that it breaks up as a direct sum of subcomplexes determined by the $V(\mathbf{m})$. This is exactly analogous to the decomposition in Lemma \ref{C-quotient-lemma}, where the quotient was the direct sum of $2^{p-1}$ subcomplexes determined by choice of signs for the first $p-1$ elements. Next, each subcomplex is shown to be isomorphic to the suspension of a new complex, $\mathcal{M}^t$. Finally, $\mathcal{M}^t$ is related to the complex $\mathcal{D}(n-p)$, whence high acyclicity is obtained.

\begin{lemma}\label{second-calc}
Fix $m_1 > m_2 > \ldots >m_1 \geq 0$. For $t < k < m_1$:
\[ T_kV(\mathbf{m}) = V(\mathbf{m})T_{k-t}\]
\end{lemma}
\begin{proof}
We proceed via induction on $t$. If $t=0$, $V(\mathbf{m})=1$, so there is nothing to show. For $t>1$, the first case of Lemma \ref{v-facemap-calc} ($T_k = T_{k+1,k}$, $m_1 > k$) shows that $T_kU_{m_1} = U_{m_1}T_{k-1}$. We may then apply the inductive hypothesis to $T_{k-1}$ and $V(m_2,\ldots,m_t)$, since $k < m_1$ implies $k-1 < m_2$. Thus:
\begin{align*}
    T_kV(\mathbf{m}) &= T_kU_{m_1}V(m_2,\ldots,m_t)\\
                    &= U_{m_1}T_{k-1}V(m_2,\ldots,m_t)\\
                    &= U_{m_1}V(m_2,\ldots,m_t)T_{k-t} = V(\mathbf{m})T_{k-t}\qedhere
\end{align*}
\end{proof}

\begin{prop}\label{quotient-disc}
For $r\geq p -1$, $(F_p/F_{p-1})_r$ is a free $R$ module with basis
\[\bigsqcup_{p \ = \ m_1 > m_2 >\ldots > m_t \geq 1}\{T_xV(\mathbf{m} + n - 1 -r) \otimes 1 : x \in (X^{A_{n-1}}_{\{s_{1+t},\ldots ,s_{n-r-2+t}\}})^{-1}\}\]
where $t \geq 1$, and is $0$ for $r<p-1$. Moreover, the $R$ span of each basis subset $\{T_xV(\mathbf{m} + n - 1 -r) \otimes 1 : x \in (X^{A_{n-1}}_{\{s_{1+t},\ldots ,s_{n-r-2+t}\}})^{-1}\}$ is equal to the $\mathcal{H}_n$ span of $V(\mathbf{m} + n - 1 -r) \otimes 1$.
\end{prop}
\begin{proof}
The first statement is immediate from Theorem \ref{Dn-decomp} and the definition of the filtration. For the second, the inclusion of $\{T_xV(\mathbf{m} + n - 1 -r) \otimes 1 : x \in (X^{A_{n-1}}_{\{s_{1+t},\ldots ,s_{n-r-2+t}\}})^{-1}\}$ in $\mathcal{H}_n\cdot V(\mathbf{m} + n - 1 -r) \otimes 1$ is clear.

In the other direction, by the basis theorem (Theorem \ref{basis-theorem}) for Iwahori-Hecke algebras it is enough to verify that $T_y\cdot V(\mathbf{m} + n - 1 -r) \otimes 1$ is contained in the $R$ span of this set for each $y\in \mathfrak{S_n}$. Each $y$ may be written as $y=xz$ for $x \in (X^{A_{n-1}}_{\{s_{1+t},\ldots ,s_{n-r-2+t}\}})^{-1}$, $z \in \langle s_{1+t},\ldots ,s_{n-r-2+t}\rangle $, with $x, z$ reduced so that $T_y = T_xT_z$.
Given $T_{k+t}$, with $k < n-r-1$ and $m_1\geq t>1$, we have $k+t < n - r -1 +t < m_1 + n - 1 -r$, so by Lemma \ref{second-calc}:
\[ T_{k+t}V(\mathbf{m} + n - 1 -r) = V(\mathbf{m} + n - 1 -r)T_k \]

We are working in (a quotient of) $\mathcal{HB}_n \otimes_{\mathcal{HB}_{n-r-1}} \mathbb{1}$, so $V(\mathbf{m} + n - 1 -r)T_k \otimes 1 = q(U_{m_1 + n - 1 -r} \otimes 1)$, since $k < n - r -1$. The element $T_z$ above may be written as the product of such $T_{k+t}$, so we see that
\[ T_y\cdot V(\mathbf{m} + n - 1 -r) \otimes 1 = q^{l(z)} T_xV(\mathbf{m} + n - 1 -r) \otimes 1\]
lies in the $R$ span of $\{T_xV(\mathbf{m} + n - 1 -r) \otimes 1 : x \in (X^{A_{n-1}}_{\{s_{1+t},\ldots ,s_{n-r-2+t}\}})^{-1}\}$.
\end{proof}

To understand the quotients $F_p/F_{p-1}$, it is necessary find a useful description of the differentials. Throughout it will be important to remember that $p=m_1 > \ldots   >m_t \geq 1$ implies $p\geq t$. 

\begin{lemma}\label{quotient-facemaps}
The face maps $\partial_j^r$ of $F_p/F_{p-1}$ are $\mathcal{H}_n$-linear and have the following effect on the elements $V(\mathbf{m} + n - 1 -r)$, for $p=m_1>\ldots  >m_t\geq1$:
\begin{equation*}
    \partial_j^r V(\mathbf{m} + n - 1 -r) =
    \begin{cases}
      0 & j < p \\
      T_{n-r+j, n-r+t} V(\mathbf{m} + n - 1 - (r-1)) & j \geq p
    \end{cases}
\end{equation*}
\end{lemma}
\begin{proof}
$\mathcal{H}_n$ linearity is immediate. Proposition \ref{fp-is-filtration} states that $\partial_j^r (F_p)_r \subseteq (F_{p-1})_{r-1}$ for $j<p$, so that $\partial_j^r$ is $0$ in the quotient. The calculation for the second case occurs in the proof of Proposition \ref{fp-is-filtration}.
\end{proof}
In $\mathcal{C}^\pm(n)$, the filtration quotients decomposed as the sum of $2^{p-1}$ submodules. With the above result, it is clear that this still holds in the Iwahori-Hecke setting---there is a subcomplex for each $\mathbf{m}$ with $m_1 = p$.
\begin{lemma}\label{quotient-direct-sum}
$F_p/F_{p-1}$ is the direct sum of $2^{p-1}$ $\mathcal{H}_n$-subcomplexes $M_{\mathbf{m}}$ (for $p=m_1>\ldots  >m_t\geq 1$), where  $M_{\mathbf{m}}$ is given in degree $r$ by the $\mathcal{H}_n$ span of $V(\mathbf{m} + n - 1 - r)$ .
\end{lemma}
\begin{proof}
Proposition \ref{quotient-disc} shows that $(F_p/F_{p-1})_r$ is the direct sum of the $R$ spans of the sets $\{T_xV(\mathbf{m} + n - 1 -r) \otimes 1 : x \in (X^{A_{n-1}}_{\{s_{1+t},\ldots ,s_{n-r-2+t}\}})^{-1}\}$, for $p=m_1 > \ldots   > m_t \geq 1$, and that this is in fact the direct sum of the $\mathcal{H}_n$ spans of the elements $V(\mathbf{m} + n - 1 -r)$. It remains to show that these are subcomplexes: but this is clear from Lemma \ref{quotient-facemaps}.
\end{proof}

While $F_p/F_{p-1}$ has now been decomposed into the $M_\mathbf{m}$, these subcomplexes are not easily compared to $\mathcal{D}(n-p)$. In order to clarify the situation, we now introduce new complexes $\mathcal{M}^t$, each constructed to be isomorphic to a suspension of $M_{m_1,\ldots,m_t}$.

For $m < n-t$, let $\mathcal{H}_m^t$ be the subalgebra of $\mathcal{H}_n \leq \mathcal{HB}_n$ generated by
\[  T_{1+t}, T_{2+t},\ldots,T_{m+t} \]
so that $\mathcal{H}_m^t$ is a reindexing of $\mathcal{H}_m$ by $+t$. In particular, it is a subalgebra corresponding to a parabolic subgroup of the Coxeter group of type $B_n$.

Let $\mathcal{M}^t$ be given in degree $-1 \leq r \leq n-p-1$ by
\begin{align*}
    \mathcal{M}^t_r &= \mathcal{H}_n \otimes_{\mathcal{H}_{n-r-p-1}^t} \mathbb{1}\\
    &= \mathcal{H}_n \otimes_{\langle T_{1+t},\ldots ,T_{n-r-p-2+t}\rangle } \mathbb{1}
\end{align*}
\[\partial_j^r(x\otimes y) = xT_{n-r+j, n-p-r+t} \otimes y, \ \partial^r = (-1)^pq^{-p}\sum_{j=0}^{r}(-1)^jq^{-j}\partial_j^r\]

\begin{lemma}
Fix $p=m_1 > m_2 > \ldots > m_t \geq 1$. The map 
\[\Phi: \ \Sigma^p  \mathcal{M}^t \longrightarrow M_{m_1,\ldots  ,m_t}\]
given in degree $p+r$ ($-1\leq r \leq n-p-1$) by $x\otimes 1 \mapsto xV(\mathbf{m} + n - 1 - (p+r)) \otimes 1$ is well-defined, and an isomorphism for each $r$.
\end{lemma}
\begin{proof}
To show that is well-defined in degree $p+r$, we must check that for $T_k \in \{T_{1+t},\ldots  ,T_{n-r-p-2+t}\}$, $\Phi(x T_k \otimes y) = \Phi(x \otimes T_k y)$. By Lemma \ref{second-calc}, since we have that $t<k<n-r-p-1+t \leq m_1 + n - 1 - (p+r)$, we obtain:
\begin{align*}
    \Phi(xT_k \otimes y) &= xT_k V(\mathbf{m} + n - 1 -(p+r)) \otimes y\\
    &= x  V(\mathbf{m} + n - 1 -(p+r))T_{k-t}\otimes y \\
    &= q(xV(\mathbf{m} + n - 1 -(p+r))\otimes y)
\end{align*}
which is equal to $x V(\mathbf{m} + n - 1 -(p+r))\otimes T_k y = \Phi(x\otimes T_k y)$.

By Proposition \ref{Hw-free-Hj}, $(\Sigma^p \mathcal{M}^t)_{p+r} = \mathcal{H}_n \otimes_{\mathcal{H}_{n-r-p-1}^t} \mathbb{1}$ is free over $R$ with basis $\{T_x : x \in (X^{A_{n-1}}_{s_{1+t},\ldots  ,s_{n-r-p-2+t}})^{-1}\}$. This is sent under $\Phi$ to $\{T_xV(\mathbf{m} + n - 1 - (p+r)) : x \in (X^{A_{n-1}}_{s_{1+t},\ldots  ,s_{n-r-p-2+t}})^{-1}\}$, which is an $R$ basis of $M_{m_1,\ldots  ,m_t}$ in degree $p+r$ by Proposition \ref{quotient-disc}. Hence $\Phi$ is an isomorphism in each degree.
\end{proof}

\begin{prop}\label{ct-iso}
The map $\Phi$ is an isomorphism of chain complexes.
\end{prop}
\begin{proof}
We show that the following diagram commutes. This will show that the differentials in $\mathcal{M}^t$ are well-defined, and that $\Phi$ is an isomorphism of chain complexes.

\[\begin{tikzcd}
	{(\Sigma^p \mathcal{M}^t)_{p+r}} &&&& {(M_{\mathbf{m}})_{p+r}} \\
	\\
	\\
	{(\Sigma^p \mathcal{M}^t)_{p+r-1}} &&&& {(M_{\mathbf{m}})_{p+r-1}}
	\arrow["{\Phi_{p+r}}", from=1-1, to=1-5]
	\arrow["{\partial^r_j}"', from=1-1, to=4-1]
	\arrow["{\Phi_{p+r-1}}"', from=4-1, to=4-5]
	\arrow["{\partial^{p+r}_{p+j}}", from=1-5, to=4-5]
\end{tikzcd}\]

$\Phi$ is $\mathcal{H}_n$ linear, so it suffices to consider $1 \otimes 1$. In degree $p+r$, going clockwise:
\begin{align*}
    1\otimes 1 &\mapsto V(\mathbf{m} - n - 1 - (p+r)) \otimes 1 \\
    &\mapsto \partial_{p+j}^{p+r}V(\mathbf{m} - n - 1 - (p+r)) \otimes 1\\
    &=\begin{cases}
        T_{n-r+j, n-p-r+t} V(\mathbf{m} + n - 1 - (p+r-1)) & j\geq 0 \\
        0 & j < 0
    \end{cases}
\end{align*}
by Lemma \ref{quotient-facemaps}. Going anticlockwise, $\partial^r_j = 0$ for $j<0$, so the diagram commutes for $j<0$. 
For $j>0$:
\begin{align*}
    1\otimes 1 &\mapsto T_{n-r+j, n-p-r+t} \otimes 1 \\
    &\mapsto T_{n-r+j, n-p-r+t} V(\mathbf{m} + n - 1 - (p+r-1))
\end{align*}
as required. Thus the face maps defined above for $\mathcal{M}^t$ correspond to the face maps in $F_p/F_{p-1}$, so are well-defined. The differential in $F_p/F_{p-1}$ then corresponds to
\[ \sum_{j=0}^{p-1} (-1)^jq^{-j} \cdot 0 + \sum_{j=p}^{p+r} (-1)^jq^{-j} \partial_{j-p}^r \]
\[ = (-1)^pq^{-p}\sum_{j=0}^r (-1)^jq^{-j} \partial_{j}^r\]
which is the $\partial^r$ defined for $\mathcal{M}^t$---thus this is a chain map, and $\Phi$ is an isomorphism of chain complexes.
\end{proof}

Putting everything thus far together:

\begin{prop}\label{binomial-decomp}
There is an isomorphism of $\mathcal{H}_n$-modules
\[F_p/F_{p-1} \cong \bigoplus_{t=1}^{p} \binom{p-1}{t-1} \ \Sigma^p \mathcal{M}^t \]
\end{prop}
\begin{proof}
There are exactly $\binom{p-1}{t-1}$ submodules of the form $M_{m_1,\ldots  ,m_t}$ with $p=m_1>\ldots  >m_t\geq1$ for fixed $t$, with $1 \leq t \leq p$. By Lemma \ref{quotient-direct-sum}, $F_p/F_{p-1}$ is the direct sum over all of these submodules. Then by Proposition \ref{ct-iso}, each $M_{m_1,\ldots  ,m_t}$ is isomorphic to the suspension of $\mathcal{M}^t$.
\end{proof}

Proposition \ref{binomial-decomp} motivates the development up to this point: with this result, it remains only to prove that $H_d(\mathcal{M}^t) = 0$ for $d\geq n - p -2$. Note that the leading factor of $(-1)^pq^{-p}$ in the definition of $\partial^r$ does not affect the homology of $\mathcal{M}^t$, since $q \in R^\times$, so we omit it from now.

 Referring back to Lemma \ref{C-quotient-lemma}, $\Sigma^p \mathcal{M}^t$ is playing the part of $M_\mathbf{v}$, and $\mathcal{D}_n$ that of $\mathcal{C}_n$. The naive isomorphism to look for, in imitation of \ref{C-quotient-lemma}, is then between $\mathcal{M}^t$ and $\mathcal{H}_n \otimes_{\mathcal{H}_{n-p}} \mathcal{D}(n-p)$. But in degree $r$, $\mathcal{M}^t$ is the module
$\mathcal{H}_n \otimes_{\mathcal{H}_{n-r-p-1}^t} \mathbb{1}$, whereas $\mathcal{H}_n \otimes_{\mathcal{H}_{n-p}} \mathcal{D}(n-p)$ is
\[ \mathcal{H}_n \otimes_{\mathcal{H}_{n-p}} (\mathcal{H}_{n-p} \otimes_{H_{n-p-r-1}} \mathbb{1}) \cong  \mathcal{H}_n \otimes_{\mathcal{H}_{n-p-r-1}} \mathbb{1} \]

In the latter, the tensor product is over $\mathcal{H}_{n-p-r-1}$, and in the former over $\mathcal{H}_{n-p-r-1}^t$. We thus consider instead a reindexed version of $\mathcal{D}(n-p)$ to account for this difference.

Let $\mathcal{D}^t(n-p)$ be the complex given in degree $r$ by
\[\mathcal{D}^t(n-p)_r = \mathcal{H}^t_{n-p} \otimes_{\mathcal{H}^t_{n-p-r-1}} \mathbb{1}\]
\[\partial^r_j(x\otimes y) = xT_{n-p-r+j+t, n-p-r+t} \otimes y, \quad \partial^r = \sum_{j=0}^r (-1)^j q^{-j}\partial^r_j \]

This is simply $\mathcal{D}(n-p)$ with everything reindexed by a factor of $+t$, so evidently  $\mathcal{D}^t(n-p) \cong \mathcal{D}(n-p)$ as chain complexes of $R$-modules, and $\mathcal{D}^t(n-p)$ is an $\mathcal{H}^t_{n-p}$ module.

Instead of $\mathcal{H}_n \otimes_{\mathcal{H}_{n-p}} \mathcal{D}(n-p)$, consider now $\mathcal{H}_n \otimes_{\mathcal{H}^t_{n-p}} \mathcal{D}^t(n-p)$. In degree $r$, this is the module:
\[ \mathcal{H}_n \otimes_{\mathcal{H}^t_{n-p}} (\mathcal{H}^t_{n-p} \otimes_{\mathcal{H}^t_{n-p-r-1}} \mathbb{1}) \cong  \mathcal{H}_n \otimes_{\mathcal{H}^t_{n-p-r-1}} \mathbb{1} \]
the same as $\mathcal{M}^t$. Unfortunately, the face maps in $\mathcal{H}_n \otimes_{\mathcal{H}^t_{n-p}} \mathcal{D}^t(n-p)$ involve right multiplication by $T_{n-p-r+j+t, n-p-r+t}$, while those in $\mathcal{M}^t$ involve right multiplication by $T_{n-r+j, n-p-r+t}$. If we can find a degree-wise isomorphism that accounts for this difference, then we will be done:
we know that $\mathcal{D}^t(n-p) \cong \mathcal{D}(n-p)$ is highly acyclic, and $\mathcal{H}_n$ is free over $\mathcal{H}^t_{n-p}$, so  $\mathcal{H}_n \otimes_{\mathcal{H}^t_{n-p}} \mathcal{D}^t(n-p)$ is also highly acyclic.

As we shall see, right multiplication by certain elements of $\mathcal{H}_n$ is a well-defined isomorphism (over $R$) $\mathcal{H}_n \otimes_{\mathcal{H}^t_{n-p-r-1}} \mathbb{1} \xrightarrow{} \mathcal{H}_n \otimes_{\mathcal{H}^t_{n-p-r-1}} \mathbb{1}$. Therefore we look for a sequence of elements $\xi(r) \in \mathcal{H}_n$ such that right multiplication by $\xi(r)$ in degree $r$ accounts for the difference in face maps.

Let the elements $\xi(r) \in \mathcal{H}_n$, $-1\leq r < n-p$ be:
\begin{align*}
     \xi(r) \ = \ & T_{n-r-1}T_{n-r}\ldots  T_{n-1}  \\
     & \boldsymbol{\cdot} T_{n-r-2}T_{n-r-1}\ldots  T_{n-2}   \boldsymbol{\cdot} \ \ldots  \ldots  \ldots \\
     & \boldsymbol{\cdot} T_{n-r-(p-t)}T_{n-r-(p-t)+1}\ldots  T_{n-(p-t)} 
\end{align*}
the product of $(r+1)(p-t)$ generators, with the convention that when $r=-1$ or $p=t$ this is interpreted as the empty product $1$.

\begin{lemma}
For $\xi(r)$ defined above and $0\leq r < n-p$, $0\leq j \leq r$, $1 \leq t \leq p \leq n$:
\[ T_{n-r+j, n-r-p+t} \ \xi(r-1) = \xi(r) \ T_{n-p-r+j+t, n-r-p+t} \]
\end{lemma}
\begin{proof}
When $r=0$, we must have $j=0$ so that $T_{n-p-r+j+t, n-r-p+t}=T_{n-p+1+t,n+1-p+r}=1$, and the statement is just the definition of $\xi(1)$. Suppose then $r\geq 1$, and proceed via induction on $p-t\geq 0$.

For $p=t$, by definition $\xi(r) = 1$ for all $r$ and there is nothing to show.

Assume $p>t$ and consider the first $r$ elements of $\xi(r-1)$. Since $p-t \geq 1$ implies $n-r-1 \geq n-r-p+t$:
\[T_{n-r+j, n-r-p+t} = T_{n-r+j,n-r}T_{n-r-1}T_{n-r-1, n-r-p+t} \]
and $T_{n-r-1, n-r-p+t}$ commutes with $T_{n-r}T_{n-r+1}\ldots  T_{n-1}$, so:
\[ T_{n-r+j, n-r-p+t} (T_{n-r}T_{n-r+1}\ldots  T_{n-1})\]
\[ = T_{n-r+j,n-r}(T_{n-r-1}T_{n-r}T_{n-r+1}\ldots  T_{n-1})T_{n-r-1, n-r-p+t}\]
We claim that 
\[ T_{n-r+j,n-r}(T_{n-r-1}T_{n-r}T_{n-r+1}\ldots  T_{n-1}) = (T_{n-r-1}T_{n-r}T_{n-r+1}\ldots  T_{n-1})T_{n-r+j-1,n-r-1}. \]
Indeed, applying the braid relation $T_{n-r}T_{n-r-1}T_{n-r}=T_{n-r-1}T_{n-r}T_{n-r-1}$ transforms the left hand side into
\[T_{n-r+j,n-r+1}(T_{n-r-1}T_{n-r}T_{n-r-1})(T_{n-r+1}\ldots  T_{n-1})\]
and using that $T_i$ commutes with $T_j$ for $|i-j|>1$ this becomes
\[ T_{n-r-1}T_{n-r+j,n-r+1}(T_{n-r}T_{n-r+1}\ldots  T_{n-1})T_{n-r-1}. \]
Repeatedly performing the same moves proves the claim.

We have shown that:
\begin{align*}
    T_{n-r+j, n-r-p+t} \xi(r-1) = & (T_{n-r-1}T_{n-r}\ldots  T_{n-1})T_{n-r+j-1,n-r-p+t} \\
    ( & \boldsymbol{\cdot} T_{n-r-1}T_{n-r}\ldots  T_{n-2}   \boldsymbol{\cdot} \ \ldots  \ldots  \ldots \\
     & \boldsymbol{\cdot} T_{n-r+1-(p-t)}T_{n-r-(p-t)+2}\ldots  T_{n-(p-t)}) 
\end{align*}
and may now apply the inductive hypothesis with $n' = n-1$, $p' = p - 1$ to obtain:
\begin{align*}
    T_{n-r+j, n-r-p+t} \xi(r-1) = & T_{n-r-1}T_{n-r}\ldots  T_{n-1}   \boldsymbol{\cdot} \ldots  \ldots  \ldots \\
     & \boldsymbol{\cdot} T_{n-r-(p-t)}T_{n-r-(p-t)+1}\ldots  T_{n-(p-t)}  T_{n-(p-t)} \\
     & \boldsymbol{\cdot} T_{n-p+r+j, n-r-p+t}
\end{align*}
as required. \qedhere
\end{proof}

We may finally complete the argument that multiplication by $\xi(r)$ is a well-defined isomorphism accounting for the difference in face maps.

\begin{prop}\label{mt-iso}
There is an isomorphism $\mathcal{M}^t \longrightarrow \mathcal{H}_n \otimes_{\mathcal{H}^t_{n-p}} \mathcal{D}^t(n-p)$ of chain complexes of $\mathcal{H}_n$ modules.
\end{prop}
\begin{proof}
Both complexes are concentrated in degrees $-1 \leq r \leq n-p-1$, so we work only in this range. In degree $r$, as discussed $\mathcal{M}^t$ is $\mathcal{H}_n \otimes_{\mathcal{H}^t_{n-p-r-1}} \mathbb{1}$ and $\mathcal{H}_n \otimes_{\mathcal{H}^t_{n-p}} \mathcal{D}^t(n-p)$ is $\mathcal{H}_n \otimes_{\mathcal{H}^t_{n-p}} (\mathcal{H}^t_{n-p} \otimes_{\mathcal{H}^t_{n-p-r-1}} \mathbb{1}) \cong  \mathcal{H}_n \otimes_{\mathcal{H}^t_{n-p-r-1}} \mathbb{1}$. We thus define our isomorphism $\Psi$ in degree $r$  by:
\[\Psi_r: \ x \otimes y \longmapsto x \ \xi(r) \otimes y\]

To show that this is well-defined, both sides of this map are $\mathcal{H}_n \otimes_{\mathcal{H}^t_{n-p-r-1}} \mathbb{1}$ so it suffices to show that $xT_k \otimes y$ and $x \otimes T_ky$ have the same image when $1+t \leq k \leq n-p-r-2+t$. The smallest index of $T_i$ occurring in $\xi(r)$ is $n-r-p+t$, so $\xi(r)$ commutes with $T_k$, and
\[ \Psi(xT_k \otimes y) = xT_k \xi(r) \otimes y = x \xi(r) T_k \otimes y =x \xi(r) \otimes T_k y = \Psi(x \otimes T_ky) \]
as required.

Next, we show that $\Psi$ is a chain map: this is true by construction. Explicitly
\begin{align*}
    \Psi_{r-1}(\partial_j^r(1 \otimes 1))&=\Psi_{r-1}(T_{n-r+j, n-r} \otimes 1)\\
    &= T_{n-r+j, n-r-p+t}\xi(r-1) \otimes 1 \\
    &= \xi(r)T_{n-r+j-p+t, n-r -p+t} \otimes 1 \\
    &= \partial_j^r(\xi(r)\otimes 1) = \partial_j^r(\Psi_r (1\otimes 1))
\end{align*}

(Both $\Psi$ and the face maps are $\mathcal{H}_n$ linear, so only $1\otimes 1$ need be considered.)

Finally, $\Psi_r$ is an isomorphism. From the identity $q^{-1}(T_s + 1 - q)T_s = 1$, the generators of any Iwahori-Hecke algebra lie in the group of units. Consequently, $\xi(r)$ (a product of the generators) is also a unit in $\mathcal{H}_n$, so has an inverse $\xi(r)^{-1}$. Then the map from $\mathcal{H}_n \otimes_{\mathcal{H}^t_{n-p}} \mathcal{D}^t(n-p)$ to $\mathcal{M}^t$ given by $x\otimes y \mapsto x\xi(r)^{-1}\otimes y$ is well defined for the same reason as $\Psi$ (the generators appearing in $\xi(r)^{-1}$ will be the same as those in $\xi(r)$), is $\mathcal{H}_n$-linear, and is a two-sided inverse to $\Psi$.
\end{proof}

We may now conclude.

\begin{theorem}\label{D-pm-acyclic}
$H_d(\mathcal{D}^\pm(n))=0$ for $d\leq n -2$.
\end{theorem}
\begin{proof}
By Proposition \ref{ct-iso} and the discussion following it, it suffices to show that $H_d(\mathcal{M}^t) = 0$ for $d\leq n - p -2$.

By Theorem 5.2 in \cite{Hepworth2020}, $H_d(\mathcal{D}(n-p)) = 0$ for $d\leq n - p -2$. $\mathcal{D}(n-p) \cong \ \mathcal{D}^t(n-p)$ as $R$-modules, and $\mathcal{H}_n$ is free as a right $\mathcal{D}^t(n-p)$ module by Proposition \ref{Hw-free-Hj}, so using Proposition \ref{mt-iso}
\[H_d(\mathcal{M}^t) \cong H_d(\mathcal{H}_n \otimes_{\mathcal{H}^t_{n-p}} \mathcal{D}^t(n-p)) \]
\[ = \mathcal{H}_n \otimes_{\mathcal{H}^t_{n-p}} H_d(( \mathcal{D}^t(n-p))\]
is $0$ for $d \leq n - p -2$.
\end{proof}

\section{Proof of Theorem~\ref{main-thm}}
It is now possible to complete the second half of the proof of Theorem \ref{main-thm}, using the complex $\mathcal{D}^\pm (n)$ and Theorem \ref{D-pm-acyclic}. The spectral sequence argument used here is standard in proofs of homological stability (see \cite{wahl2022homological}). Hepworth (\cite{Hepworth2020}) has adapted this spectral sequence argument for the setting of homolgical stability of algebras; we use his method with almost no modification in our case of type $B_n$ algebras.

\begin{lemma}\label{spectral-sequence}
There is a homological spectral sequence $\{ E^r \}_{r\geq 1}$ with:
\begin{itemize}
    \item $E^1_{p,q}$ concentrated in horizontal degrees $p \geq -1$
    \item $E^1_{p,q} = \text{Tor}_q^{\mathcal{HB}_{n-p-1}}(\mathbb{1}, \mathbb{1})$
    \item $d^1: E^1_{p,q} \xrightarrow{} E^1_{p-1,q}$ the stabilisation map for $p$ even and vanishing for $p$ odd
    \item $E_{p,q}^\infty = 0$ in total degrees $p+q \leq n-2$
\end{itemize}
\end{lemma}
\begin{proof}
The construction of such a spectral sequence for algebra homology is the content of Section 9 in \cite{Hepworth2020}. The same argument works in this case, with $\mathcal{HB}_n$ in place of $\mathcal{H}_n$ and $\mathcal{D}^\pm(n)$ playing the role of $\mathcal{D}(n)$. In particular, we note the following:
\begin{itemize}
    \item $\mathcal{HB}_n$ is free as a right $\mathcal{HB}_{n-p-1}$-module by Proposition \ref{Hw-free-Hj}.
    \item The analysis of $\Xi_*^{-1} \circ d^1 \circ \Xi_*$  remains the same as in \cite{Hepworth2020}, since the differentials in $\mathcal{D}^\pm(n)$ match those in $\mathcal{D}(n)$.
    \item $T_{n-p+j, n-p}$ commutes with $\mathcal{HB}_{n-p-1}$.\qedhere
\end{itemize}
\end{proof}

The argument now concludes by a concise application of $\{ E^r \}$ and the stated properties. Recall that Theorem \ref{main-thm} states that
\[ \text{Tor}^{\mathcal{HB}_{n-1}}_d(\mathbb{1}, \mathbb{1}) \longrightarrow \text{Tor}^{\mathcal{HB}_{n}}_d(\mathbb{1}, \mathbb{1}) \]
is an isomorphism for $2d \leq n-1$.

\begin{proof}[Proof of Theorem~\ref{main-thm}]
By induction on $n$, noting that $n=1,2$ hold trivially. Suppose then $n\geq 3$. We wish to show that the stabilisation map is an isomorphism. By Lemma \ref{spectral-sequence}:
\[ \text{Tor}^{\mathcal{HB}_{n}}_d(\mathbb{1}, \mathbb{1}) = E^1_{-1, d} \qquad \text{Tor}^{\mathcal{HB}_{n-1}}_d(\mathbb{1}, \mathbb{1}) = E^1_{0, d} \]
and $d^1: E^1_{0, d} \xrightarrow{} E^1_{-1, d}$ is the stabilisation map. Thus the statement of the theorem is equivalent to $E^2_{-1,d} = E^2_{0,d} = 0$ for $2d \leq n-1$.

For $u\geq 1$, consider the differential $E^1_{2u, q} \xrightarrow{} E^1_{2u-1, q}$. This is the stabilisation map, so by the inductive hypothesis is an isomorphism for $2q \leq n -2u -1$. The differentials $E^1_{2u-1, q} \xrightarrow{} E^1_{2u-2, q}$ are $0$ by Lemma \ref{spectral-sequence}, hence $E^2_{2u,q} = E^2_{2u-1,q} =0$ for $2q \leq n -2u -1$. From this it follows that for $r\geq 2$:
\[ E^r_{p,q} = 0 \quad \text{in bidegrees } (p, q) \text{ with } \ p \geq 1,  \ 2q \leq n - p -2\]

Now consider which differentials may affect the terms in the columns $p=-1, 0$. For $p=-1$, the differential $d^r$ landing in bidegree $(-1, q)$ originates in $E^r_{-1+r, q-r+1}$. If $2q \leq n-1$, then $E^r_{-1+r, q-r+1} = 0$ by the above---this implies that $E^2_{-1,q} = E^\infty_{-1,q}$ for $2q \leq n-1$. For $p=0$ the differential $d^r$ landing in bidegree $(0, q)$ originates in $E^r_{r, q-r+1}$. When $2q < n-1$ or $r>2$, $2(q-r+1)\leq n-r-2$, so $E^r_{-1+r, q-r+1} = 0$ by the above. When $2q=n-1$ and $r=2$, we have $2(q-r+1)=n-3$, so $E^2_{r, q-r+1}=E^r_{r, q-r+1}=0$, since $E^2_{2u, q}=0$ for $2q\leq n-2u-1$. Hence $E^2_{0,q} = E^\infty_{0,q}$.

By  Lemma \ref{spectral-sequence}, $E^\infty_{-1,q}=E^\infty_{0,q}=0$ for $2q \leq n-1$, since $n \geq 3$. Thus $E^2_{-1,q} = E^2_{0,q} = 0$ for $2q \leq n-2$, as required. \qedhere

\end{proof}

\vspace{20mm}

\bibliography{main.bib}{}
\bibliographystyle{alpha}
\end{document}